%% file: sls_spacetime_decay.tex
\documentclass[11pt,letterpaper]{article}
\usepackage[margin=1in]{geometry}
\usepackage[T1]{fontenc}
\usepackage{amsmath,amssymb,amsfonts,mathtools,bm}
\usepackage{amsthm}
\usepackage{mathrsfs}
\usepackage{microtype}
\usepackage{enumitem}
\usepackage{xcolor}
\usepackage{booktabs}
\usepackage{graphicx}
\usepackage{float}
\usepackage{placeins}
\usepackage{tikz}
\usepackage{cite}
\usepackage{textcomp}
\usepackage{hyperref}
\usetikzlibrary{arrows.meta,fit,positioning}
\hypersetup{
  pdftitle={Spatiotemporal Response Decay for Near-Optimal Distributed LQR via System Level Synthesis},
  pdfauthor={Chenchen Zhou and Jose Matias},
  colorlinks=true,
  linkcolor=blue!55!black,
  citecolor=blue!55!black,
  urlcolor=blue!55!black
}


\newtheorem{theorem}{Theorem}
\newtheorem{lemma}{Lemma}
\newtheorem{corollary}{Corollary}
\newtheorem{proposition}{Proposition}
\newtheorem{assumption}{Assumption}
\newtheorem{definition}{Definition}

\newcommand{\R}{\mathbb{R}}
\newcommand{\Z}{\mathbb{Z}}

\newcommand{\calF}{\mathcal{F}}
\newcommand{\calG}{\mathcal{G}}
\newcommand{\calI}{\mathcal{I}}
\newcommand{\calL}{\mathcal{L}}
\newcommand{\calM}{\mathcal{M}}
\newcommand{\calN}{\mathcal{N}}
\newcommand{\calS}{\mathcal{S}}
\newcommand{\calT}{\mathcal{T}}
\newcommand{\calV}{\mathcal{V}}
\newcommand{\calZ}{\mathcal{Z}}
\newcommand{\st}{\mathrm{st}}
\newcommand{\row}{\mathrm{row}}
\newcommand{\col}{\mathrm{col}}
\newcommand{\diag}{\operatorname{diag}}
\newcommand{\smin}{\sigma_{\min}}
\newcommand{\RH}{\mathcal{RH}}

\newcommand{\barHtwo}{\overline{\mathcal H}_2}
\newcommand{\norm}[1]{\left\lVert #1 \right\rVert}
\newcommand{\trans}{^{\mathsf T}}
\newcommand{\xbar}{\bar{\bm{x}}}
\newcommand{\bx}{\bm{x}}
\newcommand{\bu}{\bm{u}}
\newcommand{\bw}{\bm{w}}
\newcommand{\bA}{\bm{A}}
\newcommand{\bB}{\bm{B}}
\newcommand{\bC}{\bm{C}}
\newcommand{\bF}{\bm{F}}
\newcommand{\bK}{\bm{K}}
\newcommand{\bL}{\bm{L}}
\newcommand{\bP}{\bm{P}}
\newcommand{\bQ}{\bm{Q}}
\newcommand{\bR}{\bm{R}}
\newcommand{\bW}{\bm{W}}
\newcommand{\Phix}{\Phi_x}
\newcommand{\Phiu}{\Phi_u}

\begin{document}
\title{Spatiotemporal Response Decay for Near-Optimal Distributed LQR via System Level Synthesis}
\author{Chenchen Zhou and Jos\'{e} Matias\thanks{This work was supported by KU Leuven.}\\
\small Chemical and Biochemical Reactor Engineering and Safety (CREaS) group\\
\small Department of Chemical Engineering, KU Leuven\\
\small Jan Pieter de Nayerlaan 5, 2860 Sint-Katelijne-Waver, Belgium\\
\small \texttt{chenchen.zhou@kuleuven.be}; \texttt{jose.matias@kuleuven.be}}
\date{}
\maketitle

\begin{abstract}
Distributed control requires choosing how far information travels and how long it is retained. For networked linear quadratic regulation (LQR), we use System Level Synthesis to bound these resources for a specified performance loss per node relative to centralized control. For locally coupled systems under uniform regularity assumptions, including stabilizability and detectability, we prove that the centralized optimal state and input responses to disturbance impulses satisfy exponential decay bounds in both time and spatial distance. On networks with polynomial neighborhood growth, sufficient communication ranges and memory horizons grow logarithmically with the inverse tolerance, independently of network size. Truncating these responses gives local finite-memory filters with stability and performance guarantees under a computable small-gain condition. Under an additional uniform local feedback condition, we also obtain controllers that confine nominal disturbance effects in space and eliminate them in finite time. Numerical experiments illustrate how performance and disturbance containment requirements guide architecture selection.
\end{abstract}


\section{Introduction}
Large networked systems arise in power grids, transportation networks, building climate systems, and spatially distributed industrial processes.  Their controllers must coordinate physically coupled subsystems without requiring every local unit to communicate with, store, and process information from the entire network.  Centralized LQR provides a natural performance benchmark, but its optimal feedback is generally global.  A scalable design must therefore quantify how much local information and memory are sufficient to approach centralized performance, rather than only require that a controller be distributed \cite{andreasson2014distributed,patel2019economic}.

Distributed controller synthesis is difficult under arbitrary information constraints \cite{blondel2000survey}.  Tractable theories exploit spatial invariance or interconnections described by graphs \cite{bamieh2002distributed,d2003distributed,langbort2004distributed}, characterize convex information structures through quadratic or sparsity invariance \cite{rotkowitz2006characterization,lessard2016convexity,furieri2020sparsity}, or impose locality and delay directly on disturbance responses through System Level Synthesis (SLS) \cite{doyle2017tutorial,wang2019system,wang2018separable,wang2019localized,anderson2019sls}.  Regularization and joint controller and control architecture design can also select among candidate links, devices, or service options within a prescribed family \cite{matni2016regularization,zardini2021codesign,zhou2026codeca}.  These theories make structured synthesis and control architecture search tractable.  A remaining question is how to choose which signals each local controller receives and how long it retains them, while guaranteeing a prescribed performance loss relative to centralized control.  An overly restrictive choice can substantially degrade the best achievable performance \cite{kariwala2007fundamental}.

Implementing a local controller requires specifying which signals it receives, when they become available, and how long they are retained \cite{lamperski2015optimal,wang2018separable,anderson2019sls}.  The \emph{immediate communication radius} specifies how far a signal can be transmitted for use within one control update.  The \emph{maximum spatial footprint} limits where that signal may be delivered and used over subsequent updates \cite{wang2019localized,anderson2019sls}.  The \emph{memory horizon} specifies how long past signals are retained for filtering, measured in control updates.  These three quantities impose different communication and storage requirements and can be selected separately.

The exponential decay of sensitivity in optimization problems defined on graphs offers a possible route from global optimality to such local resource choices.  In optimal control, spatially distributed LQR can inherit spatial decay from the plant data \cite{motee2008optimal}.  For finite graph local LQR, the centralized optimal gain decays with graph distance, and truncating that gain can give a distributed controller with near optimal performance \cite{shin2023near}.  More recent work studies spatially decaying network LQR and predictive SLS performance as communication is enlarged \cite{zhang2025optimal,wu2025predsls}.  Gain decay quantifies how a control action weights the current global state.  Implementation resources, however, depend on how a localized disturbance propagates through state and input trajectories over both space and time.  The cited results do not give a guarantee uniform over network size that converts a prescribed LQR performance loss into separate requirements on immediate communication radius, maximum spatial footprint, and memory horizon.

This motivates the question studied here: \emph{what immediate communication radius, maximum spatial footprint, and memory horizon are sufficient to attain a prescribed performance loss per node relative to centralized LQR?}  We study the disturbance responses of the centralized optimum.  SLS represents a state feedback controller by the maps from a disturbance impulse at one node to the resulting state and input trajectories across the network.  Each impulse response coefficient represents how a disturbance at one node affects a state or input at a given node after a given number of steps.  These indices let us quantify how much of the centralized optimal response lies outside a prescribed local space and time support.

\begin{figure}[!t]
\centering
\resizebox{0.70\linewidth}{!}{\input{Figures/spacetime_decay_tikz}}
\caption{Centralized optimal state and input responses to localized disturbance impulses.  Each layer is the plant graph at one time instant.  Red and blue identify two disturbance locations, and fading intensity indicates smaller response coefficients at later times and more distant nodes.}
\label{fig:spacetime-decay-schematic}
\end{figure}
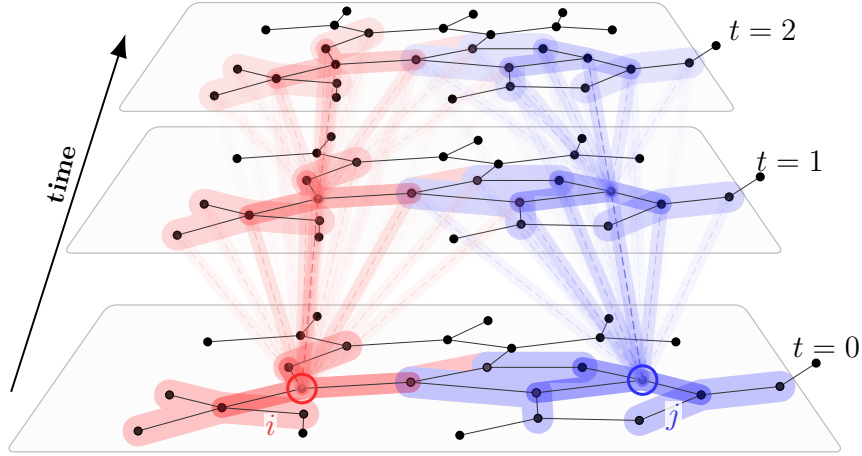

Although the centralized optimal state and input responses are generally dense, their spatial and temporal decay allows small coefficients to be discarded; Fig.~\ref{fig:spacetime-decay-schematic} illustrates this geometry.  When graph neighborhoods grow slowly enough, summing bounds on the discarded coefficients controls the total approximation error.  This connects response decay to the choice of which signals a local controller receives and how long it retains them.

Discarding coefficients generally violates the SLS equality linking the state and input responses to the plant dynamics.  We therefore develop two controller constructions.  Direct truncation implements the retained coefficients as local filters that use the specified signals and retain them for the chosen memory horizon.  The achieved closed-loop responses may extend beyond the spatial and temporal support of these filters.  Exact localized SLS reoptimizes the coefficients permitted by the chosen implementation mask subject to the SLS equality.  When feasible, the nominal response to each disturbance impulse remains within a prescribed neighborhood and vanishes after a specified horizon.  Both constructions use decay to guide resource selection; the choice also depends on whether the application requires disturbance containment.

\emph{Contributions.}
First, we establish exponential spatiotemporal decay of the centralized optimal LQR responses.  Every state and input response block decays jointly with elapsed time and with distance outside an effective coupling cone, with constants uniform over the network family (Theorem~\ref{thm:physical-light-cone}).

Second, under subexponential graph growth, we bound the accumulated magnitude and LQR energy of the discarded response coefficients (Theorem~\ref{thm:tail}).  On polynomial growth networks, these approximation bounds, together with the controller constructions below, yield a logarithmic dependence of the sufficient immediate communication radius, maximum spatial footprint, and memory horizon on the inverse tolerance for LQR loss per node.  The corresponding communication, storage, and filtering counts at each node grow polylogarithmically, with constants independent of the total network size (Corollary~\ref{cor:accuracy-resources}).  The omitted LQR energy also provides an inexpensive metric for reducing a candidate set of control architectures before localized synthesis.

Third, we establish performance guarantees uniform in network size for both constructions.  For direct truncation, a computable small gain condition certifies internal stability and yields an exponentially decaying bound on the achieved LQR cost gap per node (Theorem~\ref{thm:direct-truncation-small-gain}).  The exact guarantee uses an additional uniform local deadbeat assumption: local feedback can eliminate the effects of disturbance impulses within a uniformly bounded number of steps.  Under this assumption, a local correction restores the SLS equality within fixed spatial and temporal margins and yields an exponentially decaying bound on the exactly localized LQR cost gap per node (Theorem~\ref{thm:conic-exact-local}).  Thermal mesh experiments test these guarantees and illustrate how performance and disturbance containment requirements lead to different control architecture and implementation route choices.

\emph{Related work.}
Localized LQR optimizes performance subject to the SLS achievability equality and prescribed spatial and finite impulse response (FIR) constraints \cite{wang2019localized}.  Separable SLS decomposes such fixed-mask problems into local computations \cite{wang2018separable}.  These results provide synthesis tools once the permitted response support is specified; they do not determine its spatial and temporal scales from a prescribed performance loss.

Exact FIR localizability combines finite time elimination of disturbance effects with spatial containment and can be restrictive.  Virtually localizable and robust SLS instead allow a nonzero SLS residual and use a small gain condition to certify the resulting closed loop \cite{matni2017virtual,anderson2019sls}.  Under approximate achievability, the response achieved by the controller is generally dense.  The SLS approximation result further shows exponential convergence in the horizon when temporal decay of the centralized response is known, and outlines a spatial extension if spatial decay is assumed \cite{anderson2019sls}.  System level bases beyond FIR responses reduce the conservatism of a fixed horizon \cite{fisher2025simplepoles}.

We establish the spatial and temporal response decay needed by the SLS residual mechanism and bound the achieved LQR loss of direct truncation uniformly in network size.  Our exact result establishes when truncation residuals admit local repair within fixed margins while preserving the decay rate.  The exact guarantee uses an additional uniform local deadbeat certificate; response decay and stabilizability alone do not guarantee exact localizability.

Our decay analysis builds on exponential off-diagonal decay for inverses and functions of sparse matrices \cite{demko1984decay,benzi2007decay,benzi2013decay} and its extension to sensitivity in graph structured optimization \cite{shin2022exponential}.  Temporal decay has also been studied in long horizon optimal control \cite{xu2018exponentially}.  The closest LQR works differ in the object and scaling they characterize.  Motee and Jadbabaie develop a continuous time operator framework for infinite spatially distributed systems; the inverse closedness argument underlying the original result was subsequently qualified \cite{motee2008optimal,curtain2009comments}.  Shin et al. give a nonasymptotic decay bound for the static optimal feedback gain on finite graph local systems and analyze a gain truncation with one spatial radius \cite{shin2023near}.  Zhang et al. broaden the plant class to spatially exponentially decaying models and characterize quasi-spatial decay of both the gain and a disturbance response controller, with rates that depend on network size; Zhang et al.'s result is spatial rather than a joint law over space and time for SLS system responses \cite{zhang2025optimal}.  Predictive SLS integrates disturbance forecasts and bounds regret as a function of communication range and prediction error, rather than deriving the spatial footprint and memory of the centralized infinite horizon response \cite{wu2025predsls}.  We instead index both optimal SLS response maps by disturbance location, responding state or input location, and elapsed time, sum their discarded tails over space and time, and translate the result into separate bounds on immediate communication radius, maximum spatial footprint, and memory horizon.

Decay of distant influence also supports local computation in graph decision problems, statistical physics, and multiagent reinforcement learning \cite{gamarnik2014correlation,bandyopadhyay2008counting,qu2022scalable}; these results concern different decision variables from the SLS responses studied here.

\emph{Notation.}
Bold symbols denote stacked signals and matrices; unbolded symbols denote quantities and local blocks at one node.  For an ordered finite index set \(\calI=(i_1,\ldots,i_m)\), write \(\operatorname{col}_{i\in\calI}M_i=[M_{i_1};\cdots;M_{i_m}]\).  For any graph \(\calG=(\calV,\mathsf E_{\calG})\), \(d_{\calG}(i,j)\) is the shortest path distance, and \([a]_+:=\max\{a,0\}\) denotes the positive part of a scalar \(a\).  Unless stated otherwise, \(\norm{\cdot}\) is the induced Euclidean norm and \(\norm{\cdot}_F\) is the Frobenius norm.  Pointwise block bounds use the induced norm; with uniformly bounded local dimensions, Frobenius bounds differ only by constants.  A block matrix \(M=\{M_{ij}\}_{i,j\in\calV}\) has locality radius \(r\) on \(\calG\) if \(M_{ij}=0\) whenever \(d_{\calG}(i,j)>r\).

For a stable transfer matrix \(F(z)=\sum_{\tau\ge0}F[\tau]z^{-\tau}\), set \(\norm{F}_{\barHtwo,N}^2=N^{-1}\sum_{\tau\ge0}\norm{F[\tau]}_F^2\).  We use \(t\) for physical time and SLS response indices, reserving \(\tau\) for implementation filter taps when both appear.

\section{Networked LQR and optimal system responses}\label{sec:setup}
\subsection{Networked LQR over a graph}

We study a linear time invariant networked system in discrete time over a finite undirected graph \(\calG=(\calV,\mathsf E_{\calG})\).  Write \(N=|\calV|\), and let \(\calN_{\calG}[i]\) denote the closed neighborhood of node \(i\).  The node dynamics are
\begin{equation}\label{eq:local-plant}
  x_i(t+1)=\sum_{j\in\calN_{\calG}[i]}\left(A_{ij}x_j(t)+B_{ij}u_j(t)\right),
  \qquad i\in\calV .
\end{equation}
Here \(x_i(t)\in\R^{n_i}\) and \(u_i(t)\in\R^{m_i}\) are the state and input at node \(i\).  The matrices \(A_{ij}\in\R^{n_i\times n_j}\) and \(B_{ij}\in\R^{n_i\times m_j}\) describe the plant couplings from node \(j\) to node \(i\).

The stage performance contribution associated with node \(i\) is
\begin{equation}\label{eq:local-stage-cost}
  \begin{aligned}
  \ell_i(t):={}&
  \frac12 x_i(t)\trans
  \sum_{j\in\calN_{\calG}[i]}Q_{ij}x_j(t)\\
  &+\frac12 u_i(t)\trans
  \sum_{j\in\calN_{\calG}[i]}R_{ij}u_j(t).
  \end{aligned}
\end{equation}
Here \(Q_{ij}\in\R^{n_i\times n_j}\), \(R_{ij}\in\R^{m_i\times m_j}\), and the block symmetries \(Q_{ij}=Q_{ji}\trans\), \(R_{ij}=R_{ji}\trans\) hold.  Thus the terms outside the block diagonal in \eqref{eq:local-stage-cost} represent cost couplings between nodes.

\emph{Centralized LQR.}
The centralized LQR problem for the networked system is
\begin{equation}\label{eq:centralized-lqr}
  \begin{aligned}
  \min_{\bx,\bu}\quad &
  \frac12\sum_{t=0}^{\infty}\bx(t)\trans \bQ\bx(t)
  +\frac12\sum_{t=0}^{\infty}\bu(t)\trans \bR\bu(t)\\
  \mathrm{s.t.}\quad &
  \bx(t+1)=\bA\bx(t)+\bB\bu(t),\quad t\ge0,\\
  &\bx(0)=\bx_0 .
  \end{aligned}
\end{equation}
Here \(\bx(t):=\operatorname{col}_{i\in\calV}x_i(t)\) and \(\bu(t):=\operatorname{col}_{i\in\calV}u_i(t)\).  Also write \(\bA:=\{A_{ij}\}_{i,j\in\calV}\), \(\bB:=\{B_{ij}\}_{i,j\in\calV}\), \(\bQ:=\{Q_{ij}\}_{i,j\in\calV}\), and \(\bR:=\{R_{ij}\}_{i,j\in\calV}\), with zero blocks outside the prescribed graph couplings.  We assume \(\bQ\succeq0\) and \(\bR\succ0\).  The cost below first sums the responses to impulses in every disturbance coordinate and then normalizes by the number of nodes.

\subsection{System Level Synthesis formulation}\label{sec:sls-formulation}

To express the LQR problem in closed-loop coordinates, add the exogenous state disturbance
\begin{equation}\label{eq:plant}
  \bx(t+1)=\bA\bx(t)+\bB\bu(t)+\bw(t),
  \
  \bw(t):=\operatorname{col}_{i\in\calV}w_i(t).
\end{equation}
Let \(\Phi_x\) and \(\Phi_u\) be the closed-loop maps from \(\bw\) to \(\bx\) and \(\bu\):
\[
  \bx=\Phi_x\bw,
  \qquad
  \bu=\Phi_u\bw.
\]
Both maps are stable and strictly proper, with impulse responses
\[
  \Phix(z)=\sum_{t=1}^{\infty}\Phix[t]z^{-t},
  \qquad
  \Phiu(z)=\sum_{t=1}^{\infty}\Phiu[t]z^{-t}.
\]
The block \(\Phi_x[t]_{ij}\) maps an impulse \(w_j\) at source node \(j\) to the state \(x_i\) after \(t\) steps; \(\Phi_u[t]_{ij}\) maps the same impulse to the input \(u_i\).  Together, \((\Phi_x,\Phi_u)\) form the SLS system-response pair.  We optimize this response pair.

Stack \(\Phi:=[\Phi_x^\top\ \Phi_u^\top]^\top\) and define
\(\calZ:=\begin{bmatrix}zI-\bA&-\bB\end{bmatrix}\).  The centralized problem becomes
\begin{equation}\label{eq:sls-lqr}
\begin{aligned}
  \min_{\Phi}\quad J_N(\Phi)\
  \mathrm{s.t.}\quad \calZ\Phi=I,\
  \Phi\in z^{-1}\RH_\infty .
\end{aligned}
\end{equation}
The equality \(\calZ\Phi=I\) is the standard state-feedback SLS achievability condition: it selects exactly the state and input maps generated by internally stabilizing state feedback \cite{wang2019system,anderson2019sls}.  The objective is
\begin{equation}\label{eq:JN-phi-def}
  J_N(\Phi):=
  \frac12
  \left\|
  \begin{bmatrix}\bQ^{1/2}&0\\0&\bR^{1/2}\end{bmatrix}
  \Phi
  \right\|_{\barHtwo,N}^2 .
\end{equation}
The objective \(J_N(\Phi)\) is the weighted state and input energy generated by unit impulses in all coordinates of \(\bw\), normalized per network node.  The factor \(1/N\) permits comparison across network sizes and does not change the optimizer for fixed \(N\).  Let
$
  \Phi^\star
  :=\operatorname{col}(\Phi_x^\star,\Phi_u^\star)
$
denote an optimal solution of \eqref{eq:sls-lqr}, and write \(J_N^\star:=J_N(\Phi^\star)\).  Thus, for \(t\ge1\), \(\Phi_x^\star[t]_{ij}\) and \(\Phi_u^\star[t]_{ij}\) are the centralized optimal state and input response blocks from an impulse at node \(j\) to node \(i\) at response index \(t\).

Locality is imposed on the system-response pair \(\Phi\).  For a prescribed linear subspace \(\calS\) of admissible response blocks indexed by disturbance location, response location, and time, localized SLS uses the exact problem \cite{wang2019localized,wang2018separable}
\begin{equation}\label{eq:localized-sls-lqr-template}
\begin{aligned}
  \min_{\Phi}\quad J_N(\Phi)\
  \mathrm{s.t.}\quad \calZ\Phi=I,\
  \Phi\in z^{-1}\RH_\infty\cap\calS .
\end{aligned}
\end{equation}
For example, requiring \(\Phi_x[t]_{ij}=0\) prevents an impulse \(w_j\) from affecting \(x_i\) after \(t\) steps.  Because \eqref{eq:localized-sls-lqr-template} retains the achievability equality, every feasible point is an exact closed-loop response pair with the prescribed locality, delay, or FIR constraints.

A feasible response pair is implemented by reconstructing an internal disturbance signal and generating the control input as \cite{wang2019system,anderson2019sls}
\[
  \widehat\bw=\bx-\widehat\bx,
  \qquad
  \widehat\bx=F_x\widehat\bw,
  \qquad
  \bu=F_u\widehat\bw,
\]
where
\begin{equation}\label{eq:sls-implementation-filters}
  F_x:=z\Phi_x-I,
  \qquad
  F_u:=z\Phi_u .
\end{equation}
Here \(\widehat\bx\) and \(\widehat\bw\) are internal controller signals.  The signal \(\widehat\bx\) is the component of the current state predicted from past reconstructed disturbances, and \(\widehat\bw=\bx-\widehat\bx\) reconstructs the new disturbance entering the plant.  In this realization, the hats on \(\bx\) and \(\bw\) therefore identify internal realization signals; they do not denote truncation or an approximation of \(\Phi\).  The realization filters \(F_x\) and \(F_u\) are determined by \(\Phi\), not optimized independently.  Achievability gives \(\Phi_x[1]=I\), so \(F_x\) is strictly causal and \(F_u\) is causal.  Their tap indices satisfy \(F_x[0]=0\), \(F_x[\tau]=\Phi_x[\tau+1]\) for \(\tau\ge1\), and \(F_u[\tau]=\Phi_u[\tau+1]\).

\section{Spatiotemporal decay of centralized optimal responses}\label{sec:response-decay}

Consider the centralized optimal response to a disturbance localized at node \(j\).  With \(t\) denoting time along the corresponding LQR trajectory, the state and input responses at node \(i\) are \(\Phi_x^\star[t+1]_{ij}\) and \(\Phi_u^\star[t+1]_{ij}\).  The shift by one follows from the strictly proper response convention in Section~\ref{sec:setup}.  We show that these blocks decay exponentially with elapsed time and with distance beyond an effective coupling cone determined by the local dynamics and cost couplings.  The decay constants remain independent of the network size under the following assumptions.

The first assumption fixes the spatial range of the plant and cost couplings and bounds the local block dimensions and matrix norms independently of the network size.

\begin{assumption}[Graph locality and uniform bounds]\label{ass:lqr-local}
There exist constants \(L>1\), \(\gamma\in(0,1)\), and integers \(r_0,d_{\rm blk}\ge1\), not depending on the network size \(N\), such that:
\begin{enumerate}[label=(\alph*)]
  \item The local state and input dimensions satisfy \(n_i,m_i\le d_{\rm blk}\) for all \(i\in\calV\).
  \item The matrices \(\bA,\bB,\bQ,\bR\) are local within radius \(r_0\) on \(\calG\); that is, their \((i,j)\) blocks vanish whenever \(d_{\calG}(i,j)>r_0\).  Moreover, \(\norm{\bA},\norm{\bB},\norm{\bQ},\norm{\bR}\le L\).
  \item \(\bQ\succeq0\) and \(\bR\succeq\gamma I\).
\end{enumerate}
\end{assumption}

Graph locality alone does not prevent stabilizability and detectability from deteriorating as the network grows.  We therefore require both properties uniformly over the network family.
\begin{assumption}[Uniform stabilizability and detectability]\label{ass:lqr-regularity}
With \(L\) as in Assumption~\ref{ass:lqr-local}, there exists \(\alpha\in(0,1)\), not depending on the network size \(N\), such that:
\begin{enumerate}[label=(\alph*)]
  \item There is a static stabilizing state feedback gain \(\bK_s\) with \(\norm{\bK_s}\le L\) and
  \begin{equation}\label{eq:uniform-stab-gain}
    \norm{(\bA-\bB\bK_s)^t}\le L\alpha^t,
    \qquad t\ge0 .
  \end{equation}
  \item With \(\bC_Q:=\bQ^{1/2}\), there is a static detectability gain \(\bL_s\) with \(\norm{\bL_s}\le L\) and
  \begin{equation}\label{eq:uniform-det-gain}
    \norm{(\bA-\bL_s\bC_Q)^t}\le L\alpha^t,
    \qquad t\ge0 .
  \end{equation}
\end{enumerate}
\end{assumption}

The conditions \(\bQ\succeq0\), \(\bR\succ0\), stabilizability, and detectability are standard in LQR \cite{bertsekas2012dynamic}.  Their uniform versions over a family of systems structured by graphs are used to obtain LQR decay and performance bounds that do not deteriorate with the network size \cite{shin2023near}.  Assumption~\ref{ass:lqr-local} additionally fixes the coupling range and local block dimensions, whereas Assumption~\ref{ass:lqr-regularity} makes the stabilizability and detectability bounds uniform.  Together, they provide the uniform locality and regularity required for the decay result below.  The gains \(\bK_s\) and \(\bL_s\) are auxiliary certificates for the system family; they are not subject to the localized architecture constraints introduced later.

To state the centralized optimal response decay in physical graph coordinates, let \(v_{\rm cone}\ge1\) bound the graph distance spanned by any nonzero block of \(\bA,\bB,\bQ\), or \(\bR\).  We call \(d_{\calG}(i,j)\le v_{\rm cone}t\) the effective coupling cone.  Under Assumption~\ref{ass:lqr-local}, one may take \(v_{\rm cone}=r_0\).

The following theorem is the main structural result.  It bounds every centralized optimal response block directly in physical distance and time.
\begin{theorem}[Physical light cone response decay]\label{thm:physical-light-cone}
Under Assumptions~\ref{ass:lqr-local} and~\ref{ass:lqr-regularity}, there exist constants \(C_{\rm lc}<\infty\), \(\rho_{\rm lc}\in(0,1)\), and \(c_{\rm lc}>0\), independent of \(N\), such that
\begin{equation}\label{eq:physical-light-cone-decay}
  \norm{\Phix^\star[t+1]_{ij}}+\norm{\Phiu^\star[t+1]_{ij}}
  \le
  C_{\rm lc}\,
  \rho_{\rm lc}^{\,t+c_{\rm lc}[d_{\calG}(i,j)-v_{\rm cone}t]_+}
\end{equation}
for all \(i,j\in\calV\) and all \(t\ge0\).
\end{theorem}

A proof is given in Appendix~\ref{app:kkt-decay-proof}.  The exponent in \eqref{eq:physical-light-cone-decay} separates two effects.  The term \(t\) gives temporal decay, while \([d_{\calG}(i,j)-v_{\rm cone}t]_+\) measures how far node \(i\) lies beyond the effective coupling cone at time \(t\).  Thus, although the centralized optimal response may be dense, its coefficients become exponentially small with elapsed time and with distance outside that cone.  For nearest neighbor couplings, \(v_{\rm cone}=1\), so the cone expands by one graph hop per time step.

\section{From optimal response decay to control architecture selection}\label{sec:architecture-selection}

\subsection{Spatiotemporal control architecture}

To connect the spatiotemporal decay in Theorem~\ref{thm:physical-light-cone} with controller implementation, we first specify which realization filter blocks \(F_{x,ij}[\tau]\) and \(F_{u,ij}[\tau]\) may be nonzero as a function of graph distance and tap index.  Operationally, node \(j\) forms its local disturbance estimate \(\widehat w_j\) and attaches a time stamp.  A permitted nonzero block \(F_{x,ij}[\tau]\) or \(F_{u,ij}[\tau]\) requires that estimate to reach node \(i\) no later than tap \(\tau\).  The permitted blocks indexed by disturbance node \(j\), implementation node \(i\), and tap \(\tau\) therefore specify the message schedule and local filtering work of the realization.

\begin{figure}[!t]
\centering
\resizebox{0.68\linewidth}{!}{%
\input{Figures/three_parameter_architecture_tikz.tex}
}
\caption{Implementation tap mask \(\calM_\theta\) induced by \(\theta=(\kappa,\bar\kappa,T)\) on a path.  At tap \(\tau=0\), a disturbance estimate may reach nodes within \(\kappa\) hops.  The permitted reach expands by \(v_{\rm cone}\) hops per sample until it reaches \(\bar\kappa\), and all coefficients vanish after tap \(T\).}
\label{fig:three-parameter-architecture}
\end{figure}
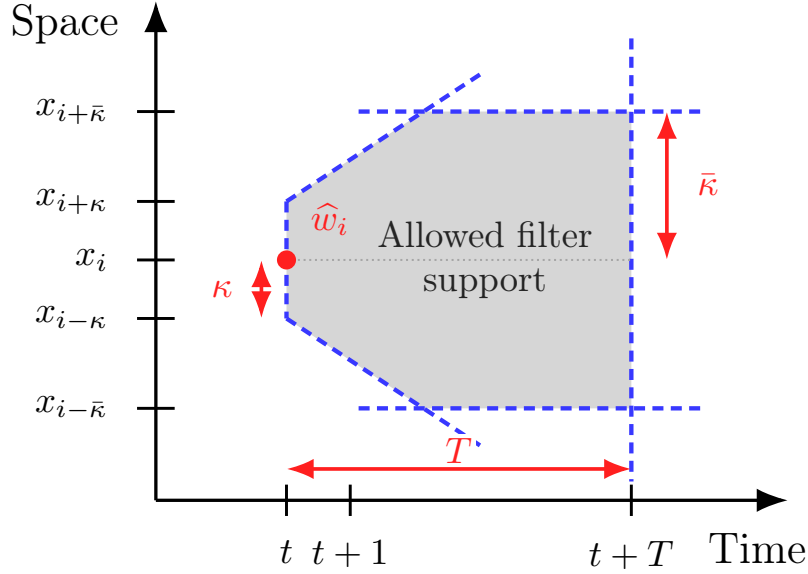

We parameterize the control architecture for localized SLS by \(\theta=(\kappa,\bar\kappa,T)\): the immediate communication radius, maximum spatial footprint, and memory horizon.  These parameters constrain individual filter taps, not the cumulative influence of an external disturbance through successive feedback updates.  Let
\[
  \Theta:=\{(\kappa,\bar\kappa,T)\in\Z_{\ge0}^3:0\le\kappa\le\bar\kappa\}
\]
and define the permitted distance at tap \(\tau\) by
\begin{equation}\label{eq:dmax-theta-def}
  d_{\max}(\tau):=\min\{\bar\kappa,\kappa+v_{\rm cone}\tau\}.
\end{equation}
The allowed tap set is
\begin{equation}\label{eq:M-kappa-kbar-T-def}
  \calM_\theta:=
  \{(\tau,i,j):0\le\tau\le T,\ d_{\calG}(i,j)\le d_{\max}(\tau)\}.
\end{equation}
For example, on a nearest neighbor path with \(v_{\rm cone}=1\), \(\theta=(1,2,2)\) permits distance one at tap zero, distance two at taps one and two, and no later coefficients; Fig.~\ref{fig:three-parameter-architecture} shows the resulting allowed triples in \(\calM_\theta\).

For a stable causal transfer matrix \(G(z)=\sum_{\tau\ge0}G[\tau]z^{-\tau}\), write \(G\in\calS_\theta^{\rm filt}\) if
\begin{equation}\label{eq:impl-support-def}
  G_{ij}[\tau]=0
  \qquad\text{whenever}\qquad
  (\tau,i,j)\notin\calM_\theta .
\end{equation}
The implementation tap mask induces the response subspace
\begin{equation}\label{eq:stacked-impl-support-def}
  \Phi\in\calS_\theta^{\rm impl}
  \quad\Longleftrightarrow\quad
  z\Phi_x-I\in\calS_\theta^{\rm filt},
  \quad z\Phi_u\in\calS_\theta^{\rm filt}.
\end{equation}
Choosing \(\calS=\calS_\theta^{\rm impl}\) in \eqref{eq:localized-sls-lqr-template} restricts exact SLS synthesis to \(\Phi\in\calS_\theta^{\rm impl}\).  The implementation tap mask \(\calM_\theta\) specifies which realization filter blocks may be nonzero; when the constrained problem is feasible, synthesis chooses their values.

The three components of \(\theta\) determine local resources.  Define the average immediate halo, maximum footprint, and stored tap count by
\begin{align}
  C_{\rm halo}(\kappa)
  &:=\frac1N\sum_{i\in\calV}|\{j:d_{\calG}(i,j)\le\kappa\}|,\label{eq:C-halo-def}\\
  C_{\rm foot}(\bar\kappa)
  &:=\frac1N\sum_{i\in\calV}|\{j:d_{\calG}(i,j)\le\bar\kappa\}|,\label{eq:C-foot-def}\\
  C_{\rm tap}(\theta)
  &:=\frac1N\sum_{i\in\calV}\sum_{\tau=0}^{T}
  |\{j:d_{\calG}(i,j)\le d_{\max}(\tau)\}|.\label{eq:C-tap-def}
\end{align}
The counts measure permitted filter-block locations; memory and arithmetic also depend on local block dimensions.  Before synthesis, the timing parameters must satisfy measured or guaranteed communication and computation deadlines.  Fixed multi-sample delays are represented by the implementation tap mask; jitter or packet loss requires conservative deadlines or an additional network uncertainty model.

\subsection{Accumulated response tails and control architecture selection}

Theorem~\ref{thm:physical-light-cone} bounds one response block at one node and time.  To measure all centralized response coefficients excluded by \(\calM_\theta\), we first define the corresponding truncation.  The mask constrains the realization filter taps \(z\Phi_x-I\) and \(z\Phi_u\), rather than \(\Phi\) directly, so the truncation must preserve the leading state response coefficient \(\Phi_x[1]=I\).  For \(\theta=(\kappa,\bar\kappa,T)\in\Theta\) and any stable transfer matrix \(G(z)=\sum_{\tau\ge0}G[\tau]z^{-\tau}\), define the projection onto \(\calM_\theta\) by
\begin{equation}\label{eq:tap-projection-def}
  [\Pi_{\calM_\theta}G]_{ij}[\tau]
  :=
  \begin{cases}
  G_{ij}[\tau],&(\tau,i,j)\in\calM_\theta,\\
  0,&\text{otherwise}.
  \end{cases}
\end{equation}
The truncation of an SLS response pair \(\Phi\) induced by \(\calM_\theta\) is
\begin{equation}\label{eq:impl-truncation-def}
  \widehat\Phi_\theta
  :=\calT_\theta^{\rm impl}\Phi
\end{equation}
as the unique stacked transfer matrix satisfying
\begin{equation}\label{eq:impl-truncation-filters}
\begin{aligned}
  (\widehat\Phi_\theta)_x[1]=I,\qquad
  z(\widehat\Phi_\theta)_x-I
  &=
  \Pi_{\calM_\theta}(z\Phi_x-I),\\
  z(\widehat\Phi_\theta)_u
  &=
  \Pi_{\calM_\theta}(z\Phi_u).
\end{aligned}
\end{equation}
Then \(\widehat\Phi_\theta\in\calS_\theta^{\rm impl}\), and \(\Phi-\widehat\Phi_\theta\) contains the response coefficients whose corresponding realization filter taps lie outside \(\calM_\theta\).  We measure the omitted centralized optimal response using an accumulated row norm and an energy norm weighted by the LQR objective.

For a block transfer matrix \(F\), define
\[
\norm{F}_{\calL_1,\row}
:=\max_a\sum_{\tau\ge0}\sum_b\norm{F[\tau]_{ab}}_F,
\]
and define \(\norm{F}_{\calL_1,\col}\) by interchanging rows and columns.  For a static matrix, omit the time sum.  Thus the row norm is the largest total magnitude received by one block row, while the column norm is the largest total magnitude generated by one source block.
\begin{equation}\label{eq:e-tail-def}
\begin{aligned}
  \widehat\Phi_\theta^{\star}
  &:=
  \calT_\theta^{\rm impl}\Phi^\star,
  \\
  E_{\rm row}(\theta)
  &:=
  \norm{\Phi^\star-\widehat\Phi_\theta^{\star}}_{\calL_1,\row},
  \\
  E_{\rm tail}(\theta)
  &:=
  \norm{\bW(\Phi^\star-\widehat\Phi_\theta^{\star})}_{\barHtwo,N},
\end{aligned}
\end{equation}
where $\bW:=\diag(\bQ^{1/2},\bR^{1/2})$.
The row tail \(E_{\rm row}(\theta)\) is the largest accumulated magnitude in any block row of the omitted response and controls the affine residual used by the direct route.  The LQR weighted response tail \(E_{\rm tail}(\theta)\) measures the omitted response energy before a controller realization is constructed; it is not the cost gap of the achieved closed loop.

Both metrics accumulate coefficients over response locations and filter taps.  The pointwise decay in Theorem~\ref{thm:physical-light-cone} therefore yields bounds independent of \(N\) only if the number of response locations at each distance is controlled uniformly over the graph family.
\begin{assumption}[Subexponential graph growth]\label{ass:growth}
There exists a subexponential function \(p:\Z_{\ge0}\to\R_{>0}\), independent of \(N\), such that
\begin{equation}\label{eq:growth}
  \left|\{j\in\calV:d_{\calG}(i,j)=d\}\right|\le p(d),
  \qquad i\in\calV,\ d\ge0 .
\end{equation}
Equivalently, for every \(\nu\in(0,1)\), \(\sup_{d\ge0}p(d)\nu^d<\infty\).
\end{assumption}

Assumption~\ref{ass:growth} is a standard sufficient condition for converting pointwise spatial decay into uniform row and column tails \cite{shin2023near}.  Fixed-dimensional grids satisfy the assumption because the number of nodes at distance \(d\) grows polynomially.  A family of regular trees does not satisfy it because the shell size grows exponentially, although a rate-dependent tail bound may still hold when the response decay dominates the branching rate.

The following theorem converts the pointwise response decay into uniform bounds for the two accumulated tail metrics.

\begin{theorem}[Response tail bound for the three parameter implementation tap mask]\label{thm:tail}
Under Assumptions~\ref{ass:lqr-local}, \ref{ass:lqr-regularity}, and~\ref{ass:growth}, let the masks \(\{\calM_\theta\}_{\theta\in\Theta}\) use the cone speed \(v_{\rm cone}\) from Theorem~\ref{thm:physical-light-cone}.  Then there exist finite constants \(C_{\rm row}>0\) and \(C_{\rm tail}>0\), and \(\delta_\kappa,\delta_{\bar\kappa},\delta_T\in(0,1)\), independent of \(N\) and \(\theta\), such that every \(\theta=(\kappa,\bar\kappa,T)\in\Theta\) satisfies
\begin{equation}\label{eq:row-tail-impl-kbarT-main}
  E_{\rm row}(\theta)
  \le
  C_{\rm row}
  \left(
  \delta_\kappa^\kappa
  +\delta_{\bar\kappa}^{\bar\kappa}
  +\delta_T^T
  \right),
\end{equation}
\begin{equation}\label{eq:weighted-tail-impl-kbarT}
  E_{\rm tail}(\theta)
  \le
  C_{\rm tail}
  \left(
  \delta_\kappa^\kappa
  +\delta_{\bar\kappa}^{\bar\kappa}
  +\delta_T^T
  \right).
\end{equation}
\end{theorem}
The proof is in Appendix~\ref{app:mask-tail-bounds}.  A coefficient is removed for one of three reasons: it lies beyond the cone set by the immediate communication radius \(\kappa\), outside the maximum spatial footprint \(\bar\kappa\), or after the memory horizon \(T\).  The physical light cone response bound \eqref{eq:physical-light-cone-decay} makes each discarded set exponentially small, and Assumption~\ref{ass:growth} prevents the number of discarded blocks from cancelling this decay.  Thus \(\kappa\), \(\bar\kappa\), and \(T\) control the tails beyond the cone, beyond the maximum footprint, and after the memory horizon, respectively.  The bounds \eqref{eq:row-tail-impl-kbarT-main} and \eqref{eq:weighted-tail-impl-kbarT} quantify approximation error; the truncated response need not satisfy the SLS achievability equality.

For a prescribed centralized response tail tolerance, choosing \(\kappa\), \(\bar\kappa\), and \(T\) proportional to \(\log(1/\varepsilon)\) makes \(E_{\rm tail}(\theta)\le\varepsilon\), with constants independent of \(N\).  For the explicit resource count, suppose that there exist constants \(C_{\rm gr}\ge1\) and \(D\ge0\), independent of \(N\), such that every graph ball of radius \(r\) contains at most \(C_{\rm gr}(1+r)^D\) nodes.  The corresponding resource growth is
\[
C_{\rm halo},C_{\rm foot}=O(\log^D(1/\varepsilon)),
\qquad
C_{\rm tap}=O(\log^{D+1}(1/\varepsilon)),
\]
with constants independent of \(N\).

\section{Localized realizations and near-optimality}\label{sec:exact-local}

For a chosen implementation tap mask, we establish stability and performance guarantees for direct truncation and exact localized SLS.

\subsection{Direct local FIR implementation}

Let \(\widehat\Phi_\theta^{\star}:=\calT_\theta^{\rm impl}\Phi^\star\) be the truncated centralized optimal response defined in \eqref{eq:impl-truncation-def}--\eqref{eq:impl-truncation-filters}.  Its realization filters satisfy the implementation tap mask \(\calM_\theta\), but \(\widehat\Phi_\theta^{\star}\) generally violates the SLS achievability equality by
\begin{equation}\label{eq:exact-section-conic-residual}
  \Delta_\theta
  :=
  \calZ\widehat\Phi_\theta^{\star}-I
  =
  -\calZ(\Phi^\star-\widehat\Phi_\theta^{\star}) .
\end{equation}
For numerical certification, define the computable residual bound
\begin{equation}\label{eq:gamma-theta-def}
  \gamma_\theta
  :=
  \left(
  \norm{\Delta_\theta}_{\calL_1,\row}
  \norm{\Delta_\theta}_{\calL_1,\col}
  \right)^{1/2}
  \ge \norm{\Delta_\theta}_{\mathcal H_\infty}.
\end{equation}
The inequality is the standard induced norm bound obtained from the maximum block row and column sums.  Because \(\Delta_\theta\) is stable and FIR, \(\gamma_\theta<1\) implies
\((I+\Delta_\theta)^{-1}=\sum_{k=0}^{\infty}(-\Delta_\theta)^k\in\RH_\infty\).
It therefore gives a sufficient, directly computable test for internal stability of the direct SLS realization.  The stronger uniform condition \(\gamma_\theta\le\alpha_{\rm sg}<1\) implies the hypothesis of Theorem~\ref{thm:direct-truncation-small-gain}.
The direct controller uses the disturbance-based realization in Section~\ref{sec:sls-formulation} with the truncated filters in \eqref{eq:impl-truncation-filters}.
Whenever \(I+\Delta_\theta\) is stably invertible, the direct realization achieves the closed-loop response pair
\begin{equation}\label{eq:direct-truncation-cl-response}
  \widehat\Phi_\theta^{\rm cl}
  :=
  \widehat\Phi_\theta^{\star}
  (I+\Delta_\theta)^{-1}.
\end{equation}
The implemented filters satisfy \(\calM_\theta\), but \eqref{eq:direct-truncation-cl-response} generally has spatial and temporal tails.  The inverse \((I+\Delta_\theta)^{-1}\) describes the closed loop; it is not an additional filter requiring global communication.  Physical coupling and repeated local feedback propagate disturbance effects through the existing connections.  In particular, the internal signal \(\widehat\bw\) can contain effects of earlier disturbances, so using only taps up to \(T\) does not imply that an external disturbance disappears after that horizon.

The following theorem converts a computable truncation-residual bound into internal stability and an LQR cost guarantee for the actual achieved response.
\begin{theorem}[Achieved response gap under direct truncation]\label{thm:direct-truncation-small-gain}
Under Assumptions~\ref{ass:lqr-local}, \ref{ass:lqr-regularity}, and~\ref{ass:growth}, let the implementation tap mask \(\calM_\theta\) use the cone speed \(v_{\rm cone}\) from Theorem~\ref{thm:physical-light-cone}.  Fix a small gain margin \(\alpha_{\rm sg}\in(0,1)\).  Then there exist \(C_{\rm cl}<\infty\) and \(\delta_\kappa,\delta_{\bar\kappa},\delta_T\in(0,1)\), independent of \(N\) and \(\theta\), such that every \(\theta=(\kappa,\bar\kappa,T)\in\Theta\) with
\(\norm{\Delta_\theta}_{\mathcal H_\infty}\le\alpha_{\rm sg}\) has an internally stable direct truncated realization.  Its achieved response \(\widehat\Phi_\theta^{\rm cl}\) satisfies
\begin{equation}\label{eq:direct-truncation-exp-gap}
  0
  \le
  J_N(\widehat\Phi_\theta^{\rm cl})-J_N^\star
  \le
  C_{\rm cl}
  \left(
  \delta_\kappa^{2\kappa}
  +\delta_{\bar\kappa}^{2\bar\kappa}
  +\delta_T^{2T}
  \right).
\end{equation}
\end{theorem}

The proof is in Appendix~\ref{app:direct-truncation-proof}.  The response tails bound the SLS residual (Lemma~\ref{lem:aux-tail-residual-estimates}), and the small gain condition controls its effect on the achieved response.  The LQR optimality identity then makes the cost gap quadratic in the achieved response error.  Thus, the response approximation bounds yield the squared decay rates in the achieved LQR cost gap.

The family also includes static gain truncation.  If the centralized feedback is \(\bu=-\bK^\star\bx\), then \(\Phi_u^\star[1]=-\bK^\star\).  At \(T=0\), the direct filters reduce to \(F_x=0\) and \(F_u\) equal to the negative of the \(\kappa\)-truncated gain, recovering the static construction in \cite{shin2023near}.  This candidate uses current measurements without storing past signals.

\subsection{Exact localized SLS reoptimization}\label{sec:exact-sls-reoptimization}
\subsubsection{Reoptimization on a prescribed mask}

Exact localized SLS requires both \(\Phi\in\calS_\theta^{\rm impl}\) and the zero-residual equality \(\calZ\Phi=I\).  Their conjunction is the standard SLS localizability requirement for a prescribed spatial and FIR response constraint \cite{wang2019localized,anderson2019sls}.  For \(\theta\in\Theta\), define
\begin{equation}\label{eq:F-kbarT-def}
  \calF_\theta^{\rm impl}:=\left\{
  \Phi\in z^{-1}\RH_\infty:
  \calZ\Phi=I,\ \Phi\in\calS_\theta^{\rm impl}
  \right\}.
\end{equation}
The minimum LQR cost over \(\calF_\theta^{\rm impl}\) is
\begin{equation}\label{eq:J-theta-loc-def}
  J_\theta^{\rm loc}
  :=
  \inf_{\Phi\in\calF_\theta^{\rm impl}} J_N(\Phi),
\end{equation}
  where \(J_\theta^{\rm loc}=+\infty\) if \(\calF_\theta^{\rm impl}\) is empty.  The SLS achievability equality and implementation tap constraints are affine, and the LQR objective is quadratic.  Because \(\calS_\theta^{\rm impl}\) permits only finitely many taps, \eqref{eq:J-theta-loc-def} is a convex quadratic program (QP).  For the state feedback objective and implementation tap mask used here, the QP separates by disturbance source.  It reoptimizes all permitted response coefficients and, when feasible, confines the actual nominal responses: \(\Phi_x[t]_{ij}\) and \(\Phi_u[t]_{ij}\) vanish for \(t>T+1\), and for \(1\le t\le T+1\) whenever \(d_{\calG}(i,j)>d_{\max}(t-1)\).  Thus an impulse has no effect outside the prescribed region or after the specified response horizon.  The shift by one follows from \eqref{eq:sls-implementation-filters}.

\subsubsection{Uniform local deadbeat repair certificate}\label{sec:deadbeat-repair}

For exact localized SLS, we obtain the resource law by repairing truncation residuals without losing the response tail rate.  A small residual alone does not ensure such repair.  The following uniform local deadbeat certificate supplies a bounded correction with fixed spatial and temporal margins.

\begin{assumption}[Uniform local deadbeat certificate]\label{ass:local-deadbeat}
There exist integers \(q\ge1\), \(r_d,r_K\ge0\), and \(L_d<\infty\), common to the network family and independent of \(N\), such that every network instance admits a feedback gain \(K_d\) satisfying
\[
  A_d:=\bA-\bB K_d,
  \qquad A_d^q=0.
\]
The matrices \(A_d\) and \(K_d\) have locality radii \(r_d\) and \(r_K\), respectively.  Their block row and block column \(\calL_1\) norms, together with those of \(A_d^s\), \(0\le s<q\), are bounded by \(L_d\).
\end{assumption}

Assumption~\ref{ass:local-deadbeat} is a sufficient, not necessary, certificate for exact localizability.  It is stronger than stabilizability and more restrictive than feasibility of one prescribed localized FIR SLS problem because it requires a static, uniformly local deadbeat gain with cancellation time and norm bounds independent of network size.  For example, an unactuated stable mode that is not nilpotent may decay asymptotically but cannot be eliminated in finite time.  The gain \(K_d\) is an auxiliary repair certificate, not the controller implemented by the repaired optimal response.  A candidate \(K_d\) certifies one network instance; validity over the network family additionally requires the common bounds in Assumption~\ref{ass:local-deadbeat}.

The following proposition states precisely what the deadbeat certificate supplies: a uniformly bounded local right inverse for strictly proper FIR residuals.
\begin{proposition}[Uniform residual repair with fixed margins]\label{prop:uniform-fir-repair}
Under Assumption~\ref{ass:local-deadbeat}, define
\begin{equation}\label{eq:constructive-repair-operator}
\begin{aligned}
  \mathcal R_x(D)
  &:=(zI-A_d)^{-1}D
    =\sum_{s=0}^{q-1}A_d^s z^{-(s+1)}D,\\
  \mathcal R_u(D)&:=-K_d\mathcal R_x(D),
  \qquad
  \mathcal R(D):=\operatorname{col}(\mathcal R_x(D),\mathcal R_u(D)).
\end{aligned}
\end{equation}
For every strictly proper FIR residual \(D\), the correction \(\mathcal R(D)\) is strictly proper and FIR and satisfies
\begin{equation}\label{eq:constructive-repair-bound}
\begin{aligned}
  \calZ\mathcal R(D)&=D,\\
  \norm{\mathcal R(D)}_{\calL_1,\row}
  &\le C_{\rm rep}\norm{D}_{\calL_1,\row},\\
  \norm{\mathcal R(D)}_{\calL_1,\col}
  &\le C_{\rm rep}\norm{D}_{\calL_1,\col},
\end{aligned}
\end{equation}
where \(C_{\rm rep}<\infty\) is independent of \(N\).  Relative to the support of \(D\), the support of \(\mathcal R(D)\) expands by at most \((q-1)r_d+r_K\) graph hops and \(q\) taps.
\end{proposition}

The proof is in Appendix~\ref{app:exact-implementation-cost}.  Reserving the repair expansion and the plant coupling margin gives the following exact feasibility and performance guarantee.

\begin{theorem}[Exact localized SLS gap under a uniform local deadbeat certificate]\label{thm:conic-exact-local}
Under Assumptions~\ref{ass:lqr-local}--\ref{ass:growth} and~\ref{ass:local-deadbeat}, let the implementation tap mask \(\calM_\theta\) use the cone speed \(v_{\rm cone}\) from Theorem~\ref{thm:physical-light-cone}.  Then there exist fixed integers \(c_{\rm sp},c_T\ge0\), constants \(C_J<\infty\), and \(\delta_\kappa,\delta_{\bar\kappa},\delta_T\in(0,1)\), all independent of \(N\) and \(\theta\), such that every \(\theta=(\kappa,\bar\kappa,T)\in\Theta\) with \(\kappa,\bar\kappa\ge c_{\rm sp}\) and \(T\ge c_T\) is exactly feasible and satisfies
\begin{equation}\label{eq:conic-exact-local-cost}
  0\le J_\theta^{\rm loc}-J_N^\star
  \le
  C_J
  \left(
  \delta_\kappa^{2\kappa}
  +\delta_{\bar\kappa}^{2\bar\kappa}
  +\delta_T^{2T}
  \right).
\end{equation}
\end{theorem}

The proof in Appendix~\ref{app:exact-implementation-cost} repairs a truncation formed inside the reserved spatial and temporal margins, then applies the response tail bound and LQR optimality identity.  The deadbeat certificate supplies eventual exact feasibility, while response decay gives the squared performance rates.  Feasibility below the fixed margins is not asserted.

\subsection{Resources independent of network size for prescribed performance}\label{sec:resource-law}

For any prescribed LQR loss per node relative to centralized control, sufficient local resources scale with the loss tolerance rather than with the total network size.
\begin{corollary}[Resources independent of network size for prescribed performance]\label{cor:accuracy-resources}
Suppose Assumptions~\ref{ass:lqr-local}--\ref{ass:growth} hold, and let \(v_{\rm cone}\) be as in Theorem~\ref{thm:physical-light-cone}.  Assume in addition that every graph ball satisfies
\[
  |\{j:d_{\calG}(i,j)\le r\}|\le C_{\rm gr}(1+r)^D
\]
  with constants \(C_{\rm gr}\ge1\) and \(D\ge0\) independent of \(N\).  Then there are constants \(c_0,c_1<\infty\), independent of \(N\), such that for every \(\varepsilon\in(0,1)\) one can choose \(\theta_\varepsilon=(\kappa_\varepsilon,\bar\kappa_\varepsilon,T_\varepsilon)\) with
\begin{equation}\label{eq:accuracy-parameter-scaling}
  \max\{\kappa_\varepsilon,\bar\kappa_\varepsilon,T_\varepsilon\}
  \le c_0+c_1\log(1/\varepsilon)
\end{equation}
such that the direct FIR realization is internally stable and
\begin{equation}\label{eq:accuracy-direct-gap}
  J_N(\widehat\Phi_{\theta_\varepsilon}^{\rm cl})-J_N^\star
  \le\varepsilon.
\end{equation}
If Assumption~\ref{ass:local-deadbeat} also holds, \(\theta_\varepsilon\) can be enlarged by fixed margins so that the exact localized problem is feasible and
\begin{equation}\label{eq:accuracy-exact-gap}
  J_{\theta_\varepsilon}^{\rm loc}-J_N^\star\le\varepsilon.
\end{equation}
For the direct route, and for the exact route when Assumption~\ref{ass:local-deadbeat} holds, the local resources satisfy
\begin{equation}\label{eq:accuracy-resource-scaling}
  C_{\rm halo},C_{\rm foot}=O(\log^D(1/\varepsilon)),
  \qquad
  C_{\rm tap}=O(\log^{D+1}(1/\varepsilon)),
\end{equation}
with constants independent of \(N\).
\end{corollary}

The proof is in Appendix~\ref{app:resource-law-proof}.  Exponential response tail bounds give logarithmic bounds on the immediate communication radius, maximum spatial footprint, and memory horizon, and polynomial graph growth converts these parameters into the resource counts in \eqref{eq:accuracy-resource-scaling}.  The direct conclusion does not require Assumption~\ref{ass:local-deadbeat}; the exact conclusion does.

\subsection{Choosing a control architecture and an implementation route}\label{sec:design-selection}

Figure~\ref{fig:mask-route-workflow} summarizes the design procedure for prescribed resources, performance, and nominal disturbance containment.
\begin{figure}[!t]
\centering
\includegraphics[width=0.64\linewidth]{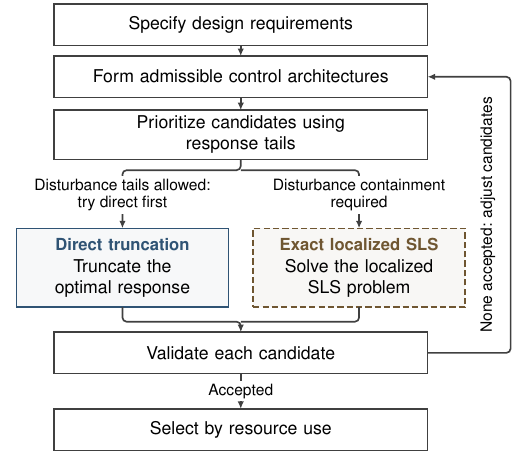}
\caption{Control architecture and implementation route selection.  Each return to candidate adjustment respects the specified resource and disturbance containment limits.}
\label{fig:mask-route-workflow}
\end{figure}

Admissible architectures satisfy communication deadlines, communication and storage budgets at each node, and any disturbance containment limits.  At comparable resource use, a smaller \(E_{\rm tail}\), computed from \(\Phi^\star\), gives higher evaluation priority.  Inexpensive candidates such as \(T=0\) remain worth testing even when their tails are large.

For direct truncation, verify stability using \(\gamma_\theta<1\) or a separate test and evaluate the achieved loss using \eqref{eq:direct-truncation-cl-response}.  For exact SLS, verify feasibility and loss of the complete localized problem.  Select the accepted design with the smallest specified resource measure, such as \(C_{\rm tap}\).  If none meets the requirements, increase \(\kappa\), \(\bar\kappa\), or \(T\) to the next admissible value within the deadlines, budgets, and containment limits.  Rank new candidates by tail reduction per added resource and repeat the checks.  If no admissible enlargement remains, report that no tested design has been validated.

\section{Numerical Experiments}\label{sec:mesh-validation}
\begin{figure}[!t]
\centering
\includegraphics[width=\linewidth]{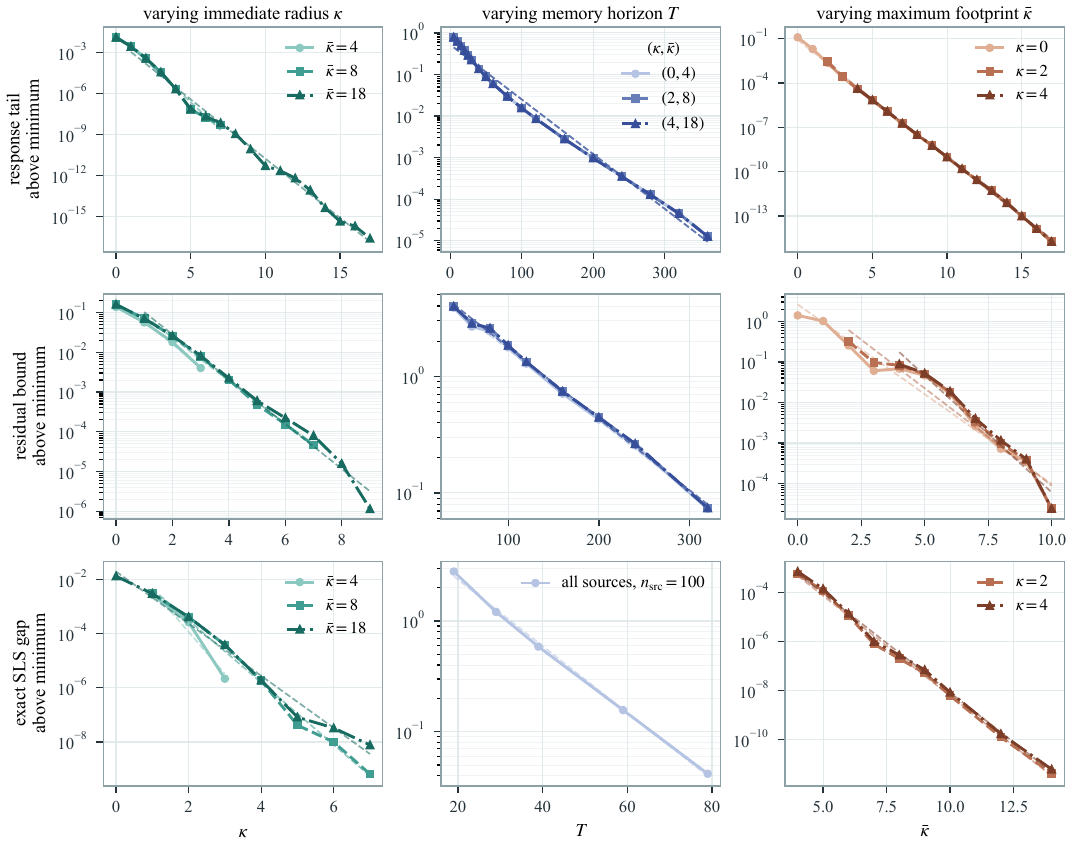}
\caption{Variation above each curve's sampled minimum in the normalized response tail (top), direct residual bound (middle), and normalized exact localized SLS cost gap (bottom), as \(\kappa\), \(T\), or \(\bar\kappa\) varies.  Dashed curves are log linear fits.  The spatial sweeps use \(T=80,160,79\) in the top, middle, and bottom rows, respectively; the bottom spatial panels use six disturbance-source nodes, whereas the bottom horizon panel uses all \(100\) nodes.  Structurally infeasible exact QPs are excluded before subtracting each sampled minimum; for \(\kappa=0\), only \(\bar\kappa=18\) is feasible in the bottom-right sweep, leaving no positive excess to plot.}
\label{fig:mesh-parameter-sweeps}
\end{figure}

We consider temperature control in a multizone building, where adjacent zones exchange heat through shared boundaries \cite{patel2019economic,shin2023near}.\footnote{Code and data for reproducing the numerical results are available at \url{https://github.com/Daakuang/spacetime-sls}, archived as version~1.0.0 at \url{https://doi.org/10.5281/zenodo.23236618}.}  The objective is to reject disturbances and restore the zone temperatures to their setpoints.  We use deviation variables about this steady state.  For zone \(i\), let \(T_i\) denote its temperature deviation, \(U_i\) the integral of \(T_i\), and \(u_i\) the local heating or cooling input.  With weighted graph Laplacian \(L_{\calG}\), the sampled dynamics and nodal stage cost are
\begin{equation}\label{eq:mesh-model}
\begin{aligned}
U_i(t+1)&=U_i(t)+\Delta t\,T_i(t),\\
T_i(t+1)&=T_i(t)-\Delta t\sum_j[L_{\calG}]_{ij}T_j(t)
  +\Delta t\,\eta u_i(t),\\
\ell_i(t)&=\eta^2U_i(t)^2+u_i(t)^2.
\end{aligned}
\end{equation}
Here, \(\Delta t\) is the sampling interval, \(k_{ij}>0\) is the edge weight in \(L_{\calG}\), and \(\eta>0\) is the actuation coefficient.  Unless stated otherwise, we use a \(10\times10\) mesh with \(N=100\), \(180\) edges, a graph diameter of \(18\) hops, \(k_{ij}=0.05\), \(\Delta t=1\), and \(\eta=2^{-4}\).  Nodes are numbered from zero, row by row, with \(x_i=\col(U_i,T_i)\).

All experiments use the centralized LQR response as the reference.  For control architecture parameters \(\theta=(\kappa,\bar\kappa,T)\), define the normalized response tail
\(p_\theta:=E_{\rm tail}(\theta)^2/(2J_N^\star)\) for \(J_N^\star>0\).
For candidate screening, we evaluate \(p_\theta\) for \(0\le\kappa\le\bar\kappa\le18\) and \(19\) horizons from \(4\) to \(399\), using \(560\) response taps and the Riccati value matrix for the remaining temporal tail.
For direct truncation, stability is first tested using \(\gamma_\theta<1\) in \eqref{eq:gamma-theta-def}.
Static direct costs use a discrete Lyapunov equation; dynamic costs use finite sums of achieved closed-loop response energy, with summation length independent of controller memory.  These finite sums are numerical estimates unless a remainder bound is supplied.
The localized QPs are solved independently for each disturbance source using OSQP \cite{stellato2020osqp}, with response indices \(1,\ldots,T+1\).  We check structural feasibility separately from numerical status and SLS equality residuals.
Solving the block columns associated with all \(N\) disturbance-source nodes gives a complete controller and cost \(J_\theta^{\rm loc}\); selected disturbance-source nodes are used only for parameter sensitivity or screening diagnostics.
The experiments follow the theoretical chain from response decay to implementation and resource scaling, and then apply the selection procedure to a performance target.

\begin{figure}[!h]
\centering
\includegraphics[width=0.72\linewidth]{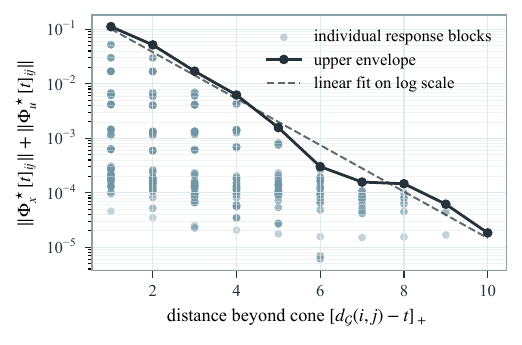}
\caption{Centralized response blocks versus distance beyond the effective coupling cone.  The solid curve is the largest norm at each distance.}
\label{fig:mesh-offcone-response}
\end{figure}

Figure~\ref{fig:mesh-offcone-response} first tests the pointwise decay statement behind the design rules.
For an impulse at the center node \(j=55\), it plots the blocks of \(\Phi_x^\star\) and \(\Phi_u^\star\) against \([d_{\calG}(i,j)-t]_+\).
The semilog envelope is consistent with the effective coupling cone in Theorem~\ref{thm:physical-light-cone}.
The accumulated response tail is the first link from pointwise decay to control architecture selection.

Figure~\ref{fig:mesh-parameter-sweeps} compares this tail with the two implementation routes under separate sweeps of \(\kappa\), \(T\), and \(\bar\kappa\).
The figure isolates sensitivity to each parameter by subtracting the minimum attained within each sampled curve.  This displays variation above the remaining baseline when the other parameters are fixed; it does not show the total loss approaching zero.  All stability, feasibility, and cost statements below use the original quantities before this subtraction.

The top row shows the normalized centralized response tail.
For \((\kappa,\bar\kappa)=(4,18)\), \(p_\theta\) falls from \(7.82\times10^{-1}\) at \(T=5\) to \(1.08\times10^{-5}\) at \(T=400\).
The observed spatial and temporal trends are consistent with the three tail terms in Theorem~\ref{thm:tail}.

The middle row tests the direct route without solving a localized QP.
For \((\kappa,\bar\kappa)=(4,18)\), \(\gamma_\theta\) falls from \(4.04\) at \(T=40\) to \(6.20\times10^{-2}\) at \(T=400\), crossing the sufficient stability threshold in Theorem~\ref{thm:direct-truncation-small-gain}.
The achieved normalized cost gap \(J_N(\widehat\Phi_\theta^{\rm cl})/J_N^\star-1\) at \(T=400\) is \(2.284\times10^{-6}\), so the induced norm certificate is conservative as a loss predictor.

The bottom row evaluates exact localized reoptimization.
For this mesh family, the deadbeat condition in Theorem~\ref{thm:conic-exact-local} can be checked in closed form.
Writing \(d:=\Delta t\), choose
\begin{equation}\label{eq:mesh-deadbeat-gain}
K_d=\begin{bmatrix}(d^2\eta)^{-1}I&(d\eta)^{-1}(2I-dL_{\calG})\end{bmatrix}.
\end{equation}
Stacking \eqref{eq:mesh-model} by state component, with \(\bm U:=\col_i U_i\) and \(\bm T:=\col_i T_i\), gives
\[
\bA-\bB K_d=\begin{bmatrix}I&dI\\-d^{-1}I&-I\end{bmatrix},
\qquad (\bA-\bB K_d)^2=0.
\]
Thus Assumption~\ref{ass:local-deadbeat} holds with \(q=2\).
Here, \(\bA-\bB K_d\) is block diagonal across graph nodes, so \(r_d=0\), whereas \(K_d\) has radius \(r_K=1\).
Since the plant radius is \(r_0=1\), the appendix construction gives \(c_{\rm sp}=r_0+(q-1)r_d+r_K=2\).
The uniformly bounded weighted degree, together with fixed \(d\) and \(\eta\), makes the repair bounds independent of mesh size.

For this mesh with \(\kappa=0\), an exact response to a disturbance at node \(j\) is feasible only if \(\bar\kappa\ge e_j\) and \(T\ge e_j+1\), where \(e_j:=\max_i d(i,j)\): the temperature propagation front precedes the permitted input, and its integral response persists one step longer.  Waiting until the source cone covers the mesh and then applying the two-step deadbeat feedback establishes sufficiency.  Thus complete exact responses require \(\bar\kappa\ge18\) and \(T\ge19\); small numerical residuals alone do not establish feasibility.
With \(\kappa=0\) and \(\bar\kappa=\kappa+v_{\rm cone}T\), the normalized exact cost gap \(J_\theta^{\rm loc}/J_N^\star-1\) over all \(100\) sources falls from \(2.86\) at \(T=19\) to \(5.53\times10^{-2}\) at \(T=99\), with block-row \(\ell_1\) SLS equality residuals below \(1.5\times10^{-11}\).
Because \(\kappa=0<c_{\rm sp}=2\), this horizon sweep violates the theorem's sufficient spatial inequalities.
The spatial panels use the six diagonal sources \(\{0,22,44,66,88,99\}\) for parameter sensitivity and include masks that both violate and satisfy these inequalities.
These finite network QPs evaluate feasibility and cost; they do not estimate the uniform repair bound.

Corollary~\ref{cor:accuracy-resources} predicts that a prescribed performance level can be maintained using local resources independent of network size.
To test this consequence without retuning, we fix \(\theta=(4,4,160)\) on \(5\times5\), \(7\times7\), \(10\times10\), and \(12\times12\) meshes.
Figure~\ref{fig:mesh-size-scaling} reports the normalized cost gap \(J_N(\widehat\Phi_\theta^{\rm cl})/J_N^\star-1\) achieved by the direct truncation realization, its residual certificate \(\gamma_\theta\), the normalized centralized response tail, and the local resource counts.
For \(25\le N\le144\), the achieved gap remains between \(6.81\times10^{-6}\) and \(2.29\times10^{-5}\), and \(0.773\le\gamma_\theta\le0.892<1\).
Meanwhile, \(C_{\rm halo}=C_{\rm foot}\le41\) and \(C_{\rm tap}\le6601\), while the average footprint falls from \(77.6\%\) to \(21.8\%\) of the network.
This finite sweep is consistent with the corollary; it does not establish uniformity or identify a minimum-resource control architecture.

\begin{figure}[t]
\centering
\includegraphics[width=0.54\linewidth]{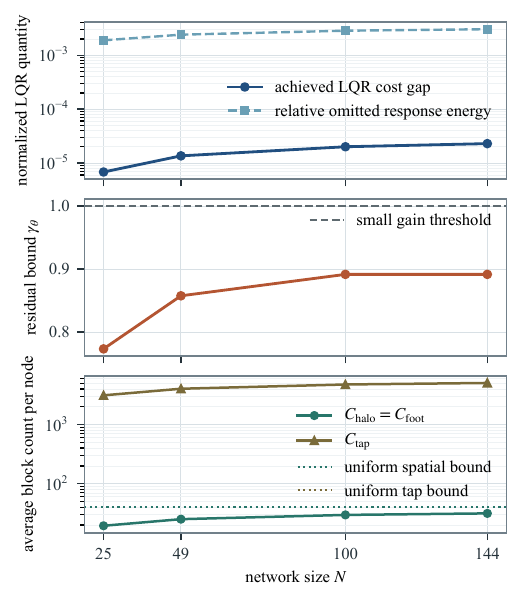}
\caption{Normalized cost gap achieved by the direct truncation realization, its residual certificate, the normalized response tail, and resource counts per node for \(\theta=(4,4,160)\) across four mesh sizes.}
\label{fig:mesh-size-scaling}
\end{figure}

Figure~\ref{fig:mesh-candidate-reduction} supports using the response tail to prioritize exact QP evaluations.
Of the \(3610\) control architectures in the sweep, \(103\) are nondominated in \((C_{\rm tap},p_\theta)\).
We evaluate control architectures both on and off this frontier, including 168 complete QPs: 140 are feasible and 28 are structurally infeasible.
Among the feasible complete QPs, the Spearman correlation between \(\log p_\theta\) and \(\log(J_\theta^{\rm loc}/J_N^\star-1)\) is \(0.990\).
The response tail thus ranks the tested control architectures well.

\begin{figure}[t]
\centering
\includegraphics[width=0.60\linewidth]{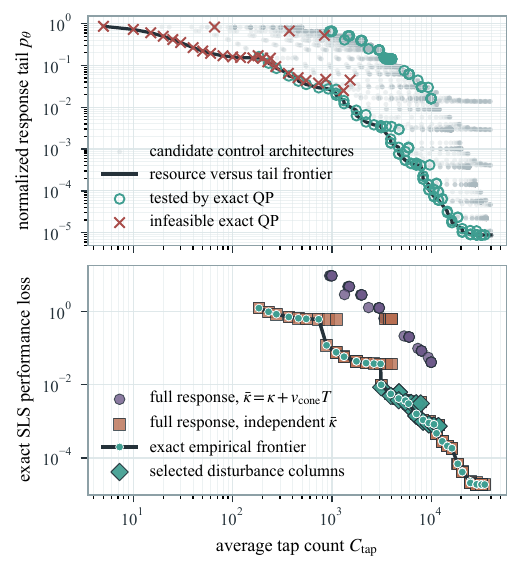}
\caption{Normalized response tail and exact localized SLS loss versus average tap count.  The lower panel distinguishes complete evaluations from evaluations using selected disturbance-source nodes.}
\label{fig:mesh-candidate-reduction}
\end{figure}

Table~\ref{tab:mesh-design-selection} compares three designs for a \(1\%\) relative LQR loss target on the same \(100\)-node mesh, with two-hop measurements available at the current update and three-hop internal signals one update later.
All costs include every disturbance coordinate.  The dynamic direct estimates sum the first \(880\) achieved closed-loop response coefficients.

For performance alone, \(\theta=(2,2,0)\) meets the target with \(0.1765\%\) loss and the smallest \(C_{\rm tap}\) among these candidates, without retaining past signals.
Its closed-loop spectral radius is \(0.9902<1\).  The dynamic candidate \((2,3,159)\) satisfies \(\gamma_\theta<1\); for \((2,2,159)\), small gain instead follows from \(\sum_t\|S\Delta_\theta[t]S^{-1}\|_2<0.683\), where \(S=I_N\otimes\diag(1/4,1)\).
Exact SLS at \(T=0\) is infeasible because the input acts only on temperature and cannot eliminate the integral response in one step.

For nominal disturbance containment within three hops and no response after index~\(160\), the exact design \(\theta=(2,3,159)\) achieves \(0.9918\%\) loss; reducing the maximum spatial footprint to two hops raises the loss to \(4.3227\%\).
Both complete QPs have block-row \(\ell_1\) SLS equality residuals below \(9.0\times10^{-12}\).

Thus, among these candidates, adding disturbance containment changes the choice from static direct control to exact SLS with a three-hop footprint and additional memory.

\begin{table}[!t]
\caption{Implementation resources and LQR loss for a \(1\%\) target.}
\label{tab:mesh-design-selection}
\centering
\setlength{\tabcolsep}{5pt}
\renewcommand{\arraystretch}{1.12}
\begin{tabular}{@{}cccc@{}}
\toprule
\(\theta\) & \(C_{\rm tap}\) & Direct loss & Exact loss\\
\midrule
\((2,2,0)\) & \(11.04\) & \(0.1765\%\) & infeasible\\
\((2,2,159)\) & \(1766.40\) & \(0.4361\%^{\dagger}\) & \(4.3227\%\)\\
\((2,3,159)\) & \(3127.44\) & \(0.0804\%^{\dagger}\) & \(0.9918\%\)\\
\bottomrule
\end{tabular}

\smallskip
\parbox{\columnwidth}{\footnotesize Loss is \(100(J/J_N^\star-1)\%\).  \(C_{\rm tap}\) is the average number of permitted filter-block locations per controller.  \(\dagger\) marks a finite-sum estimate.}
\end{table}

\FloatBarrier
\section{Conclusion}

Spatiotemporal response decay provides a quantitative basis for control architecture selection.  Under the stated assumptions, a prescribed LQR loss per node relative to centralized control can be attained with communication, storage, and filtering counts independent of network size.  The two controller constructions distinguish limits on implementation resources from limits on disturbance propagation.  Direct truncation uses local filters while allowing disturbance tails; exact localized SLS additionally confines the nominal disturbance response.  The thermal mesh illustrates how this distinction changes the control architecture and implementation route selected for a prescribed performance target.

Further work can address three limitations of the present analysis.  One direction is to replace the uniform local deadbeat assumption with weaker, verifiable conditions for local residual repair that preserve the performance and resource bounds.  Another is to quantify disturbance propagation outside the prescribed support under model uncertainty and incorporate this sensitivity into control architecture selection.  A further direction is to incorporate communication service choices, delays, and installation costs into control architecture models and develop scalable algorithms for controller and control architecture co-design.  Response decay could guide the selection of candidate control architectures within these richer models.

\appendix

\section{Auxiliary Facts and Proofs}\label{app:auxiliary-proofs}
\subsection{Finite horizon representation and optimality response graph}\label{app:kkt-graph-definition}

Each optimal response coefficient is the sensitivity of an optimal state or input trajectory to a localized initial condition.  We represent this sensitivity by a problem with a finite horizon whose terminal cost preserves the centralized infinite horizon trajectory.  Let \(\bP^\star\succeq0\) be the stabilizing Riccati value matrix for \eqref{eq:centralized-lqr}.  For horizon \(T_{\rm h}\) and initial condition \(\xbar\), consider
\begin{equation}\label{eq:finite-lqr-main}
\begin{aligned}
  \min_{\substack{\bx(0:T_{\rm h})\\ \bu(0:T_{\rm h}-1)}}\quad &
  \sum_{\tau=0}^{T_{\rm h}-1}\ell(\bx(\tau),\bu(\tau))
  +V^\star(\bx(T_{\rm h}))\\
  \mathrm{s.t.}\quad &
  \bx(0)=\xbar,\\
  &\bx(\tau+1)=\bA\bx(\tau)+\bB\bu(\tau),\\
  &\hfill \tau=0,\ldots,T_{\rm h}-1,
\end{aligned}
\end{equation}
where \(\ell(x,u):=\frac12x^\top\bQ x+\frac12u^\top\bR u\) and \(V^\star(x):=\frac12x^\top\bP^\star x\).  The Riccati terminal cost is the infinite horizon value function, so the optimizer of \eqref{eq:finite-lqr-main} agrees with the centralized infinite horizon trajectory over its first \(T_{\rm h}\) steps.  Perturbing the initial state at node \(j\) therefore produces the SLS response to a disturbance impulse at node \(j\).

The first order optimality conditions of \eqref{eq:finite-lqr-main} form a linear Karush--Kuhn--Tucker (KKT) system.  Group its state, input, and multiplier coordinates by node and time, and denote them by \(X_i(t)\), \(U_i(t)\), and \(\Lambda_i(t)\).  The multiplier \(\Lambda_j(-1)\) enforces the initial condition at source node \(j\).  The following graph records which KKT blocks can be nonzero.

\begin{definition}[KKT response graph and distance]\label{def:infinite-graph}
Let \(\mathcal K_\infty=(\mathcal V_\infty,\mathcal E_\infty)\) be the graph of the repeating, nonterminal KKT blocks associated with \eqref{eq:finite-lqr-main}.  Its vertices are
\[
  \begin{aligned}
  \mathcal V_\infty
  ={}&\{X_i(t):t\ge0\}\cup\{U_i(t):t\ge0\}\\
  &{}\cup\{\Lambda_i(-1)\}\cup\{\Lambda_i(t):t\ge0\}.
  \end{aligned}
\]
For all indices for which the vertices exist, its edges are
\[
\begin{array}{rcl}
Q_{ij}\ne0 &:& X_i(t)\sim X_j(t),\\
R_{ij}\ne0 &:& U_i(t)\sim U_j(t),\\
\text{initial constraint} &:& \Lambda_i(-1)\sim X_i(0),\\
\text{state update} &:& \Lambda_i(t)\sim X_i(t+1),\\
A_{ij}\ne0 &:& \Lambda_i(t)\sim X_j(t),\\
B_{ij}\ne0 &:& \Lambda_i(t)\sim U_j(t).
\end{array}
\]
The dense Riccati terminal Hessian is excluded so that the graph distance is independent of the finite horizon; the effect of the dense terminal Hessian \(\bP^\star\) is treated in Lemma~\ref{lem:terminal-separation}.  Assign a physical node and a time layer to each KKT coordinate by
\[
\begin{aligned}
  \operatorname{node}(X_i(t))&=\operatorname{node}(U_i(t))=i,\\
  \operatorname{node}(\Lambda_i(s))&=i,\qquad s\ge-1,\\
  \tau(X_i(t))&=\tau(U_i(t))=t,\\
  \tau(\Lambda_i(-1))&=0,\qquad
  \tau(\Lambda_i(s))=s+1,\quad s\ge0.
\end{aligned}
\]
With \(d_\infty\) denoting ordinary graph distance on \(\mathcal K_\infty\), define the state and input response distances by
\begin{align}
  d_{\st}^X((t,i),(0,j))&:=d_\infty(X_i(t),\Lambda_j(-1)),
  \notag\\
  d_{\st}^U((t,i),(0,j))&:=d_\infty(U_i(t),\Lambda_j(-1)),
  \label{eq:physical-dst-def}
\end{align}
and let
\[
  d_{\st}((t,i),(0,j))
  :=\min\{d_{\st}^X((t,i),(0,j)),d_{\st}^U((t,i),(0,j))\}.
\]
\end{definition}

Figure~\ref{fig:kkt-response-graph} illustrates this construction.
\begin{figure}[!t]
\centering
\resizebox{0.82\linewidth}{!}{\input{Figures/kkt_response_graph_definition_tikz}}
\caption{Construction of the KKT response graph in Definition~\ref{def:infinite-graph}.  (a) Nonzero cost and dynamics blocks form a cell that repeats over time.  (b) Shortest paths from \(\Lambda_j(-1)\) to \(X_i(t)\) and \(U_i(t)\) define the state and input response distances.  The arrows indicate response blocks.}
\label{fig:kkt-response-graph}
\end{figure}
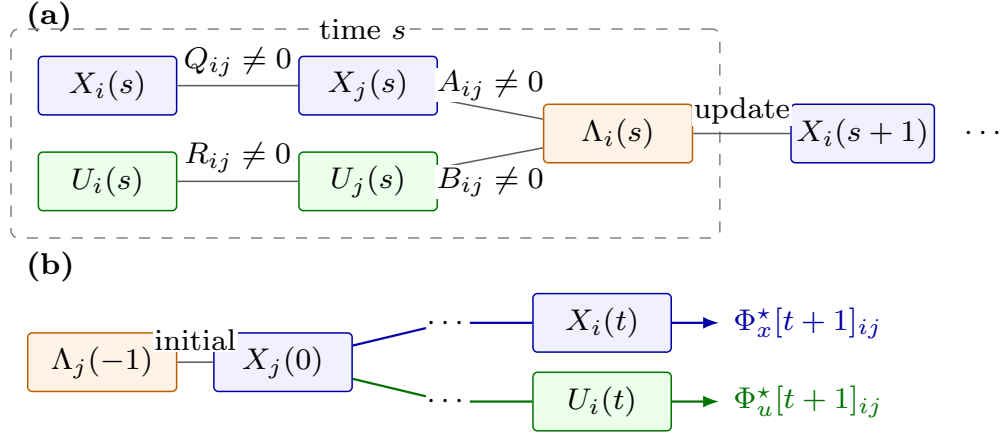

Although \(\bQ^{1/2}\) appears in the detectability condition, the KKT Hessian contains \(\bQ\).  The KKT graph is therefore determined by the sparsity of \(\bQ\), not \(\bQ^{1/2}\).

\subsection{Auxiliary facts on inverse decay}\label{app:inverse-decay}
We use the standard inverse decay principle of Demko--Moss--Smith and its graph extensions \cite{demko1984decay,benzi2007decay,benzi2013decay,shin2022exponential}.  If a self-adjoint block matrix \(H\) has nonzero blocks only along edges of a graph \(\mathcal K\) and
\[
  \sigma(H)\subset[-L,-\gamma]\cup[\gamma,L],
  \qquad 0<\gamma\le L,
\]
then there are constants \(C_H<\infty\) and \(\rho_H\in(0,1)\), depending only on \(L\), \(\gamma\), and fixed local block dimensions, such that
\begin{equation}\label{eq:inverse-decay}
  \norm{(H^{-1})_{ab}}\le C_H\rho_H^{d_{\mathcal K}(a,b)} .
\end{equation}
Inequality~\eqref{eq:inverse-decay} follows from the usual polynomial approximation proof of inverse decay.  The graph determines which blocks of low degree polynomials in \(H\) vanish, while the approximation constants are fixed by the spectral interval.  Hence graph cardinality and the degree of a particular vertex do not enter \(C_H\) or \(\rho_H\).  When the inverse decay estimate \eqref{eq:inverse-decay} is applied to the KKT graph with the Riccati terminal cost, the dense terminal Hessian is included by adding a clique at the terminal layer.  The clique may shorten graph distances to the terminal layer, but it does not change the inverse decay constants.  The separation argument below prevents the terminal clique from shortening the distance between the response coordinates of interest.

\subsection{KKT proof of spatiotemporal response decay}\label{app:kkt-decay-proof}
The proof follows the finite horizon sensitivity mechanism used for distributed LQR gain decay in \cite{shin2023near}, with two changes.  The KKT graph in Definition~\ref{def:infinite-graph} distinguishes state, input, and multiplier blocks, and the sensitivity blocks are interpreted as SLS response coefficients.

\begin{lemma}[Riccati value and uniform KKT regularity]\label{lem:riccati-existence}\label{lem:kkt-conditioning}
Under Assumptions~\ref{ass:lqr-local} and~\ref{ass:lqr-regularity}, the discrete algebraic Riccati equation for \eqref{eq:centralized-lqr} has a stabilizing solution \(\bP^\star\succeq0\) with
\[
  \norm{\bP^\star}\le L_P:=\frac{L^2(L+L^3)}{1-\alpha^2}.
\]
The centralized SLS optimum satisfies
\begin{equation}\label{eq:traceP-cost-equivalence}
  J_N^\star=J_N(\Phi^\star)=\frac{1}{2N}\operatorname{trace}(\bP^\star).
\end{equation}
Moreover, the KKT matrix \(H_{T_{\rm h}}\) associated with \eqref{eq:finite-lqr-main} satisfies
\begin{equation}\label{eq:kkt-conditioning}
  \norm{H_{T_{\rm h}}}\le L_H,\qquad
  \smin(H_{T_{\rm h}})\ge\gamma_H ,
\end{equation}
with constants independent of \(T_{\rm h}\) and \(N\).
\end{lemma}

\begin{proof}
Uniform stabilizability and detectability, together with \(\bR\succ0\), give the stabilizing Riccati solution.  Summing its value \(V^\star(e_k)=\frac12e_k^\top\bP^\star e_k\) over unit disturbance coordinates and dividing by \(N\) gives \eqref{eq:traceP-cost-equivalence}.  The value bound follows by evaluating the cost under the stabilizing certificate in Assumption~\ref{ass:lqr-regularity}:
\[
  \bP^\star\preceq\sum_{t\ge0}(\bF_s^t)^\top(\bQ+\bK_s^\top\bR\bK_s)\bF_s^t,
  \qquad \bF_s=\bA-\bB\bK_s .
\]
Uniform stabilizability gives the displayed \(L_P\).

We next establish KKT constants independent of network size and horizon.  Write the finite horizon equality constraints as \(F_{T_{\rm h}}y=b\), with \(y\) collecting state and input blocks in time order, and let \(G_{T_{\rm h}}=\diag(\bQ,\bR,\ldots,\bQ,\bR,\bP^\star)\) be the Hessian.  The upper bound \(\norm{H_{T_{\rm h}}}\le L_H\) follows from Assumption~\ref{ass:lqr-local} and \(\norm{\bP^\star}\le L_P\).  The terminal matrix \(\bP^\star\) may be dense, but only its operator norm is used in this bound; no sparsity of \(\bP^\star\) is required for conditioning.

The stabilizing certificate constructs a uniformly bounded right inverse of \(F_{T_{\rm h}}\).  Given an arbitrary residual sequence \(r_{-1},r_0,\ldots,r_{T_{\rm h}-1}\) for the initial and dynamics equalities, choose \(x(0)=r_{-1}\), \(u(t)=-\bK_s x(t)\), and \(x(t+1)=(\bA-\bB\bK_s)x(t)+r_t\).  This produces \(F_{T_{\rm h}}y=r\).  The geometric bound \eqref{eq:uniform-stab-gain} and convolution with \(\{\alpha^t\}_{t\ge0}\) give
\[
  F_{T_{\rm h}}Y_{T_{\rm h}}=I,\qquad
  \norm{Y_{T_{\rm h}}}\le c_F ,
\]
with \(c_F\) independent of \(T_{\rm h}\) and \(N\).

We next spell out the nullspace curvature estimate.  Let \(y=(x,u)\in\ker F_{T_{\rm h}}\), so \(x(0)=0\) and \(x(t+1)=\bA x(t)+\bB u(t)\).  With \(\bC_Q=\bQ^{1/2}\), write
\[
  x(t+1)=(\bA-\bL_s\bC_Q)x(t)+\bL_s\bC_Qx(t)+\bB u(t).
\]
Iterating from \(x(0)=0\) gives
\[
  x(t)=\sum_{\tau=0}^{t-1}
  (\bA-\bL_s\bC_Q)^{t-1-\tau}
  \big(\bL_s\bC_Qx(\tau)+\bB u(\tau)\big).
\]
By \eqref{eq:uniform-det-gain}, Young's inequality for causal convolution, and the bounds on \(\bL_s,\bB\), there is \(c_d<\infty\), independent of \(T_{\rm h},N\), such that
\begin{equation}\label{eq:detectability-energy-expanded}
  \norm{x}_{\ell_2(0:T_{\rm h})}
  \le
  c_d\big(
  \norm{\bC_Qx}_{\ell_2(0:T_{\rm h}-1)}
  +\norm{u}_{\ell_2(0:T_{\rm h}-1)}
  \big).
\end{equation}
Consequently,
\begin{align*}
  &\norm{x}_{\ell_2(0:T_{\rm h})}^2
  +\norm{u}_{\ell_2(0:T_{\rm h}-1)}^2\\
  &\qquad\le
  (1+2c_d^2)\big(
  \norm{\bC_Qx}_{\ell_2(0:T_{\rm h}-1)}^2
  +\norm{u}_{\ell_2(0:T_{\rm h}-1)}^2
  \big).
\end{align*}
Since \(\bR\succeq\gamma I\) and \(\bP^\star\succeq0\), the terminal contribution may be discarded and
\begin{equation}\label{eq:nullspace-curvature-expanded}
  y^\top G_{T_{\rm h}}y
  \ge
  c_G\norm{y}^2,
  \
  y\in\ker F_{T_{\rm h}},
  \
  c_G:=\frac{\min\{1,\gamma\}}{1+2c_d^2}.
\end{equation}
Thus the detectability certificate controls unpenalized state directions through \(\bQ^{1/2}\), while \(\bR\succeq\gamma I\) controls input directions; all constants are uniform in the horizon and network size.  The lower bound does not use the terminal Hessian except for its nonnegativity, so the density of \(\bP^\star\) cannot degrade the coercivity constant.  It remains to convert the right inverse and nullspace curvature into a singular value bound for the saddle point matrix
\[
  H_{T_{\rm h}}=
  \begin{bmatrix}
    G_{T_{\rm h}} & F_{T_{\rm h}}^\top\\
    F_{T_{\rm h}} & 0
  \end{bmatrix}.
\]
Since \(F_{T_{\rm h}}Y_{T_{\rm h}}=I\), \(\sigma_{\min}(F_{T_{\rm h}})\ge c_F^{-1}\).  To bound \(H_{T_{\rm h}}^{-1}\), consider
\[
  G_{T_{\rm h}}y+F_{T_{\rm h}}^\top\lambda=f,
  \qquad F_{T_{\rm h}}y=g.
\]
Set \(y_g=Y_{T_{\rm h}}g\) and \(z=y-y_g\in\ker F_{T_{\rm h}}\).  Taking the inner product of the first equation with \(z\), using \(z^\top F_{T_{\rm h}}^\top\lambda=0\), and applying \eqref{eq:nullspace-curvature-expanded} gives
\begin{equation}\label{eq:kkt-primal-bound}
  \norm{z}
  \le c_G^{-1}\big(\norm{f}+\norm{G_{T_{\rm h}}}c_F\norm{g}\big),
  \
  \norm{y}\le c_F\norm{g}+\norm{z}.
\end{equation}
Moreover, \(F_{T_{\rm h}}^\top\) has smallest singular value at least \(c_F^{-1}\), and therefore
\begin{equation}\label{eq:kkt-dual-bound}
  \norm{\lambda}
  \le c_F\norm{f-G_{T_{\rm h}}y}.
\end{equation}
Equations~\eqref{eq:kkt-primal-bound} and \eqref{eq:kkt-dual-bound}, together with the uniform bound on \(\norm{G_{T_{\rm h}}}\), show that \(\norm{H_{T_{\rm h}}^{-1}}\) depends only on \(c_F,c_G\), and \(L_H\).  Thus \(\smin(H_{T_{\rm h}})\ge\gamma_H>0\), uniformly in \(T_{\rm h},N\).  This is the regularity mechanism used in \cite{shin2023near}, written in the variables of \eqref{eq:finite-lqr-main}.  The estimate uses \(\bQ^{1/2}\) only as an energy certificate, not to define KKT sparsity.
\end{proof}

For a horizon \(T_{\rm h}\), let \(\mathcal K_{T_{\rm h}}^{\rm term}\) be the finite block graph of the full KKT matrix with the Riccati terminal cost.  It contains all repeating edges in Definition~\ref{def:infinite-graph} up to the horizon and the terminal edges \(X_i(T_{\rm h})\sim X_j(T_{\rm h})\) induced by \((\bP^\star)_{ij}\).  Let \(d_{T_{\rm h}}^{\rm term}\) be its graph distance.  By \eqref{eq:kkt-conditioning} and \eqref{eq:inverse-decay},
\begin{equation}\label{eq:uniform-H-inverse-decay}
  \norm{(H_{T_{\rm h}}^{-1})_{ab}}
  \le C_H\rho_H^{d_{T_{\rm h}}^{\rm term}(a,b)}
\end{equation}
uniformly in \(T_{\rm h}\) and \(N\).  The distance \(d_\infty\), computed without edges from the terminal Hessian, is related to the finite graph by the following separation lemma.

\begin{lemma}[Terminal separation for nonterminal KKT vertices]\label{lem:terminal-separation}
Let \(a,b\in\mathcal K_\infty\) be vertices before the terminal layer, with time layers \(\tau(a)\) and \(\tau(b)\) defined in Definition~\ref{def:infinite-graph}.  Suppose the finite graph up to \(T_{\rm h}\), before adding the terminal Hessian edges, contains \(a,b\) and at least one shortest path in \(\mathcal K_\infty\) between them.  If
\begin{equation}\label{eq:terminal-separation-condition}
  2T_{\rm h}-\tau(a)-\tau(b)-c_{\rm term}
  >
  d_\infty(a,b),
\end{equation}
where \(c_{\rm term}\) is a fixed constant determined by the offsets between KKT coordinate types, then
\begin{equation}\label{eq:terminal-separation-equality}
  d_{T_{\rm h}}^{\rm term}(a,b)=d_\infty(a,b).
\end{equation}
The constant \(c_{\rm term}\) depends only on the KKT block convention and not on \(T_{\rm h}\), \(N\), or the density of \(\bP^\star\).
\end{lemma}

\begin{proof}
Because a shortest path that avoids the terminal Hessian is present in the finite graph, \(d_{T_{\rm h}}^{\rm term}(a,b)\le d_\infty(a,b)\).  Any finite path that avoids the terminal Hessian is also a path in \(\mathcal K_\infty\), so it cannot be shorter than \(d_\infty(a,b)\).  A path that uses \(\bP^\star\) must first reach a terminal state vertex \(X_i(T_{\rm h})\), cross one dense terminal edge, and return from some \(X_j(T_{\rm h})\) to \(b\).  The list of KKT coordinate types is finite, and every edge before the terminal layer changes the layer by at most one.  Thus there is a fixed \(c_0\) such that reaching the terminal layer from any earlier vertex \(v\) requires at least \(T_{\rm h}-\tau(v)-c_0\) edges.  Every path through the terminal clique therefore has length at least
\[
  (T_{\rm h}-\tau(a)-c_0)+1+(T_{\rm h}-\tau(b)-c_0),
\]
which is at least \(2T_{\rm h}-\tau(a)-\tau(b)-c_{\rm term}\) after setting \(c_{\rm term}=2c_0+1\).  Under \eqref{eq:terminal-separation-condition} such a path is longer than \(d_\infty(a,b)\).  Hence no shortest finite path uses the terminal layer, proving \eqref{eq:terminal-separation-equality}.
\end{proof}

For every fixed pair \(a,b\), the finite graph contains the required path for all sufficiently large \(T_{\rm h}\).  For example, it is enough to take \(T_{\rm h}\ge \max\{\tau(a),\tau(b)\}+d_\infty(a,b)\).  Hence the horizon can be chosen to satisfy Lemma~\ref{lem:terminal-separation} for each response coefficient without changing that coefficient.

\begin{proposition}[Response decay in KKT distance]\label{prop:kkt-response-decay}
Under Assumptions~\ref{ass:lqr-local} and~\ref{ass:lqr-regularity}, there exist constants \(C>0\) and \(\rho\in(0,1)\), independent of \(N\), such that for all \(i,j\in\calV\) and all \(t\ge0\),
\begin{align}
  \norm{\Phix^\star[t+1]_{ij}}&\le C\rho^{d_{\st}^X((t,i),(0,j))},
  \notag\\
  \norm{\Phiu^\star[t+1]_{ij}}&\le C\rho^{d_{\st}^U((t,i),(0,j))}.
  \label{eq:typed-response-decay}
\end{align}
Consequently,
\begin{equation}\label{eq:phi-decay}
  \norm{\Phix^\star[t+1]_{ij}}+
  \norm{\Phiu^\star[t+1]_{ij}}
  \le 2C\rho^{d_{\st}((t,i),(0,j))}.
\end{equation}
\end{proposition}

\begin{proof}[Proof of Proposition~\ref{prop:kkt-response-decay}]
The Riccati terminal cost makes the finite horizon optimizer agree with the infinite horizon LQR trajectory on the first \(T_{\rm h}\) steps.  Fix \(i,j\) and \(t\), and choose \(T_{\rm h}>t\) so that the finite graph contains shortest paths from \(X_i(t)\) and \(U_i(t)\) to \(b=\Lambda_j(-1)\), and so that Lemma~\ref{lem:terminal-separation} applies to both pairs.  Perturbing the initial condition at node \(j\) affects only the initial multiplier row \(\Lambda_j(-1)\); let \(M_j\) denote the corresponding signed injection matrix, with \(\sup_j\norm{M_j}<\infty\).  Differentiating the KKT system gives
\begin{align*}
  \Phix^\star[t+1]_{ij}
  &=\Pi_{X_i(t)}H_{T_{\rm h}}^{-1}\Pi_{\Lambda_j(-1)}^\top M_j,\\
  \Phiu^\star[t+1]_{ij}
  &=\Pi_{U_i(t)}H_{T_{\rm h}}^{-1}\Pi_{\Lambda_j(-1)}^\top M_j .
\end{align*}
Thus one admissible choice of response decay constants is
\begin{equation}\label{eq:admissible-response-constants}
  \rho:=\rho_H,\qquad C:=C_H\sup_j\norm{M_j}.
\end{equation}
Using \eqref{eq:uniform-H-inverse-decay} proves \eqref{eq:typed-response-decay}--\eqref{eq:phi-decay}.
\end{proof}

\begin{proof}[Proof of Theorem~\ref{thm:physical-light-cone}]
For \(V\in\{X,U\}\), every KKT path from \(\Lambda_j(-1)\) to \(V_i(t)\) must create net time displacement \(t\) and physical displacement \(d_{\calG}(i,j)\).  Each edge in Definition~\ref{def:infinite-graph} changes the time layer by at most one and spans at most \(v_{\rm cone}\) physical graph hops.  Hence a path of length \(\ell\) satisfies \(\ell\ge t\) and \(\ell\ge d_{\calG}(i,j)/v_{\rm cone}\), so
\begin{align*}
  d_{\st}^V((t,i),(0,j))
  &\ge
  \max\{t,d_{\calG}(i,j)/v_{\rm cone}\}
  \\
  &=
  t+\frac{[d_{\calG}(i,j)-v_{\rm cone}t]_+}{v_{\rm cone}} .
\end{align*}
Substitution into Proposition~\ref{prop:kkt-response-decay} gives \eqref{eq:physical-light-cone-decay} with \(C_{\rm lc}=2C\), \(\rho_{\rm lc}=\rho\), and \(c_{\rm lc}=1/v_{\rm cone}\).
\end{proof}

\subsection{Proofs for control architecture selection from response decay}\label{app:mask-selection-proofs}
\subsubsection{Tail bounds for implementation tap masks}\label{app:mask-tail-bounds}

For \(\tau,d\in\Z_{\ge0}\), define the physical decay index
\begin{equation}\label{eq:physical-decay-index}
  s(\tau,d):=\tau+\frac{[d-v_{\rm cone}\tau]_+}{v_{\rm cone}}
  =\max\{\tau,d/v_{\rm cone}\}.
\end{equation}

\begin{lemma}[Summability of physical decay shells]\label{lem:row-slice-tail-from-growth}\label{lem:column-slice-tail-from-growth}
Under Assumption~\ref{ass:growth}, for \(q\in\Z_{\ge0}\), let \(h(q)\) be the largest number of state and input response blocks in any fixed block row or block column whose indices satisfy \(q\le s(\tau,d_{\calG}(i,j))<q+1\).  Then \(h\) is bounded by a subexponential function independent of \(N\).  Consequently, for every \(\rho\in(0,1)\) and every \(\delta\in(\rho,1)\), there is \(M_\delta<\infty\), independent of \(N\), such that
\begin{equation}\label{eq:physical-shell-tail}
  \sum_{q\ge r}h(q)\rho^q\le M_\delta\delta^r,
  \qquad r\in\Z_{\ge0}.
\end{equation}
\end{lemma}

\begin{proof}
The condition \(s(\tau,d)<q+1\) implies \(0\le\tau<q+1\) and \(0\le d<v_{\rm cone}(q+1)\).  For either a fixed response location or a fixed disturbance source, Assumption~\ref{ass:growth} therefore gives
\[
  h(q)
  \le c_{\rm type}(q+1)
  \sum_{d=0}^{\lceil v_{\rm cone}(q+1)\rceil}p(d),
\]
where \(c_{\rm type}\) accounts only for the fixed state and input coordinate types.  The right hand side is subexponential because \(p\) is subexponential.  Thus \(\bar M_\delta:=\sup_{q\ge0}h(q)(\rho/\delta)^q<\infty\).  Since \(h(q)\rho^q\le\bar M_\delta\delta^q\), summing the geometric series proves \eqref{eq:physical-shell-tail} with \(M_\delta=\bar M_\delta/(1-\delta)\).  The same count applies to rows and columns because \(d_{\calG}\) is symmetric.
\end{proof}

\begin{lemma}[Block row bound for normalized response energy]\label{lem:l1-to-h2}
Any stable block transfer matrix \(E\), with a fixed number \(n_{\rm rowtype}\) of block row types per physical node, satisfies
\begin{equation}\label{eq:l1-to-h2}
  \norm{E}_{\barHtwo,N}
  \le \sqrt{n_{\rm rowtype}}\,\norm{E}_{\calL_1,\row}.
\end{equation}
For any static \(\bW\) with \(\norm{\bW}\le L_W\),
\begin{equation}\label{eq:weighted-l1-to-h2}
  \norm{\bW E}_{\barHtwo,N}
  \le L_W\sqrt{n_{\rm rowtype}}\,\norm{E}_{\calL_1,\row}.
\end{equation}
\end{lemma}

\begin{proof}
For each block row \(a\), the squared Frobenius norm of all its blocks and taps is no larger than the square of their summed Frobenius norms.  There are at most \(n_{\rm rowtype}N\) block rows.  Summing these inequalities and dividing by \(N\) gives \eqref{eq:l1-to-h2}.  The submultiplicativity of the normalized \(\mathcal H_2\) norm gives \eqref{eq:weighted-l1-to-h2}.
\end{proof}

\begin{lemma}[Uniform bound for the SLS affine map]\label{lem:block-row-bounds}\label{lem:sls-affine-bound}
Under Assumptions~\ref{ass:lqr-local} and~\ref{ass:growth}, there is \(L_{\rm sls}<\infty\), independent of \(N\), such that
\begin{align}
  \norm{\calZ E}_{\calL_1,\row}
  &\le L_{\rm sls}\norm{E}_{\calL_1,\row},\notag\\
  \norm{\calZ E}_{\calL_1,\col}
  &\le L_{\rm sls}\norm{E}_{\calL_1,\col}.
  \label{eq:sls-affine-bounded}
\end{align}
\end{lemma}

\begin{proof}
The map \(\calZ E=(zI-\bA)E_x-\bB E_u\) consists of a time shift and multiplication by \(\bA\) and \(\bB\).  Assumption~\ref{ass:lqr-local} confines every nonzero block of these matrices to distance \(r_0\), while Assumption~\ref{ass:growth} bounds the number of such blocks in every row and column.  Their block row and column \(\calL_1\) norms are therefore bounded uniformly in \(N\).  Applying the triangle and product inequalities for these norms proves \eqref{eq:sls-affine-bounded}.
\end{proof}

\begin{lemma}[Truncation tail and residual estimates]\label{lem:truncation-fir}\label{lem:aux-tail-residual-estimates}
Under the hypotheses of Theorem~\ref{thm:tail}, the truncation induced by \(\calM_\theta\) is FIR and stable.  There exist constants and rates, independent of \(N\) and \(\theta=(\kappa,\bar\kappa,T)\), such that
\begin{equation}\label{eq:tail-impl-kbarT}
  \norm{\Phi^\star-\widehat\Phi_\theta^\star}_{\calL_1,\row}
  \le
  C_{\rm row}
  \left(\delta_\kappa^\kappa+\delta_{\bar\kappa}^{\bar\kappa}+\delta_T^T\right).
\end{equation}
The same rates can be chosen so that
\begin{equation}\label{eq:tail-impl-kbarT-col}
  \norm{\Phi^\star-\widehat\Phi_\theta^\star}_{\calL_1,\col}
  \le
  C_{\rm col}
  \left(\delta_\kappa^\kappa+\delta_{\bar\kappa}^{\bar\kappa}+\delta_T^T\right).
\end{equation}
With \(\Delta_\theta=\calZ\widehat\Phi_\theta^\star-I\),
\begin{equation}\label{eq:residual-impl-kbarT-row}
  \norm{\Delta_\theta}_{\calL_1,\row}
  \le
  C_\Delta
  \left(\delta_\kappa^\kappa+\delta_{\bar\kappa}^{\bar\kappa}+\delta_T^T\right),
\end{equation}
and the constants can be enlarged so that
\begin{equation}\label{eq:residual-impl-kbarT-schur}
  \sqrt{\norm{\Delta_\theta}_{\calL_1,\row}\norm{\Delta_\theta}_{\calL_1,\col}}
  \le
  C_{\rm sg}
  \left(\delta_\kappa^\kappa+\delta_{\bar\kappa}^{\bar\kappa}+\delta_T^T\right).
\end{equation}
\end{lemma}

\begin{proof}
The truncation keeps finitely many taps and retains \((\widehat\Phi_{\theta,x}^{\star})[1]=I\), so it is stable and FIR.  Write \(d=d_{\calG}(i,j)\).  By Theorem~\ref{thm:physical-light-cone} and \eqref{eq:physical-decay-index}, every state or input response block at implementation tap \(\tau\) is bounded by \(C_{\rm lc}\rho_{\rm lc}^{s(\tau,d)}\).  A block is discarded only if
\[
  \tau>T,
  \qquad d>\bar\kappa,
  \qquad\text{or}\qquad
  d>\kappa+v_{\rm cone}\tau.
\]
These conditions imply, respectively,
\[
  s(\tau,d)>T,
  \quad s(\tau,d)>\bar\kappa/v_{\rm cone},
  \quad s(\tau,d)>\kappa/v_{\rm cone}.
\]
Apply Lemma~\ref{lem:row-slice-tail-from-growth} and sum the bounds for the three discarded sets.  Integer rounding is absorbed into the constants.  Taking powers \(1/v_{\rm cone}\) into the spatial rates gives \eqref{eq:tail-impl-kbarT}; the column count gives \eqref{eq:tail-impl-kbarT-col}.

Let \(E_\theta=\Phi^\star-\widehat\Phi_\theta^\star\).  Since \(\calZ\Phi^\star=I\),
\[
  \Delta_\theta=-\calZ E_\theta.
\]
Lemma~\ref{lem:sls-affine-bound} applied to \eqref{eq:tail-impl-kbarT} and \eqref{eq:tail-impl-kbarT-col} proves \eqref{eq:residual-impl-kbarT-row} and the analogous column bound.  For every frequency,
\[
  \norm{\Delta_\theta(e^{\mathrm i\omega})}_2
  \le
  \sqrt{\norm{\Delta_\theta}_{\calL_1,\row}
  \norm{\Delta_\theta}_{\calL_1,\col}}.
\]
Taking the supremum over \(\omega\) and enlarging the constant proves \eqref{eq:residual-impl-kbarT-schur}.
\end{proof}

\begin{proof}[Proof of Theorem~\ref{thm:tail}]
The row tail bound \eqref{eq:row-tail-impl-kbarT-main} is \eqref{eq:tail-impl-kbarT}.  Apply \eqref{eq:weighted-l1-to-h2} to \(E=\Phi^\star-\widehat\Phi_\theta^\star\) and use \eqref{eq:tail-impl-kbarT} to obtain \eqref{eq:weighted-tail-impl-kbarT}.  Assumption~\ref{ass:lqr-local} bounds \(\norm{\bW}\), so all fixed factors are absorbed into \(C_{\rm tail}\).
\end{proof}

\subsection{Proofs for direct truncation and exact localized SLS cost gaps}\label{app:exact-repair-proofs}

\begin{lemma}[LQR orthogonality identity]\label{lem:lqr-orthogonality}
For every stable response pair \(\Phi\) satisfying \(\calZ\Phi=I\),
\begin{equation}\label{eq:pythagorean-cost}
  J_N(\Phi)-J_N^\star
  =
  \frac12\norm{\bW(\Phi-\Phi^\star)}_{\barHtwo,N}^2 .
\end{equation}
\end{lemma}

\begin{proof}
Let \(E=\Phi-\Phi^\star\).  Both responses satisfy the affine SLS equation, so \(\calZ E=0\), and \(\Phi^\star+sE\) is feasible for every scalar \(s\).  First-order optimality of \(\Phi^\star\) therefore gives
\[
  \left\langle\bW\Phi^\star,\bW E\right\rangle_{\barHtwo,N}=0.
\]
Expanding \(\norm{\bW(\Phi^\star+E)}_{\barHtwo,N}^2\) in \eqref{eq:JN-phi-def} proves \eqref{eq:pythagorean-cost}.  The standard state-feedback SLS theorem ensures that every stable solution of \(\calZ\Phi=I\) has an internally stabilizing realization \cite{wang2019system,anderson2019sls}.
\end{proof}

\subsubsection{Small gain proof for direct truncation}\label{app:direct-truncation-proof}
\begin{proof}[Proof of Theorem~\ref{thm:direct-truncation-small-gain}]
Equation \eqref{eq:residual-impl-kbarT-schur} gives
\[
  \varepsilon:=\norm{\Delta_\theta}_{\mathcal H_\infty}
  \le
  C_{\rm sg}(\delta_\kappa^\kappa+\delta_{\bar\kappa}^{\bar\kappa}+\delta_T^T).
\]
If \(\varepsilon\le\alpha_{\rm sg}<1\), set \(R_\theta=(I+\Delta_\theta)^{-1}\).  The Neumann series gives
\[
  R_\theta\in\RH_\infty,
  \qquad
  \norm{R_\theta}_{\mathcal H_\infty}
  \le(1-\varepsilon)^{-1}.
\]
The truncated realization therefore has achieved response
\[
  \widehat\Phi_\theta^{\rm cl}
  =
  \widehat\Phi_\theta^\star R_\theta,
\]
which is stable and satisfies \(\calZ\widehat\Phi_\theta^{\rm cl}=(I+\Delta_\theta)R_\theta=I\).  The achieved response pair is generally dense and is not an exact localized response.  Let \(E_\theta=\widehat\Phi_\theta^\star-\Phi^\star\).  Since \(R_\theta-I=-\Delta_\theta R_\theta\),
\begin{equation}\label{eq:direct-achieved-error-decomposition}
  \widehat\Phi_\theta^{\rm cl}-\Phi^\star
  =E_\theta R_\theta-\Phi^\star\Delta_\theta R_\theta.
\end{equation}
The \(\mathcal H_2\)-\(\mathcal H_\infty\) product inequality now gives
\[
  \norm{\bW(\widehat\Phi_\theta^{\rm cl}-\Phi^\star)}_{\barHtwo,N}
  \le
  \frac{E_{\rm tail}(\theta)+\norm{\bW\Phi^\star}_{\barHtwo,N}\varepsilon}{1-\varepsilon}.
\]
By \eqref{eq:traceP-cost-equivalence}, \(\norm{\bW\Phi^\star}_{\barHtwo,N}^2=2J_N^\star\le d_{\rm blk}L_P\).  Theorem~\ref{thm:tail} and \eqref{eq:residual-impl-kbarT-schur} therefore bound the right hand side by a constant times \(\delta_\kappa^\kappa+\delta_{\bar\kappa}^{\bar\kappa}+\delta_T^T\), while \(1-\varepsilon\ge1-\alpha_{\rm sg}\).  Apply Lemma~\ref{lem:lqr-orthogonality} and \((a+b+c)^2\le3(a^2+b^2+c^2)\) to obtain \eqref{eq:direct-truncation-exp-gap}.  The lower bound follows from centralized optimality.
\end{proof}

\subsubsection{Exact implementation and cost}\label{app:exact-implementation-cost}

\begin{proof}[Proof of Proposition~\ref{prop:uniform-fir-repair}]
Because \(A_d^q=0\), \((zI-A_d)^{-1}=\sum_{s=0}^{q-1}A_d^s z^{-(s+1)}\).  Since \(\mathcal R_u(D)=-K_d\mathcal R_x(D)\),
\[
\calZ\mathcal R(D)
=(zI-\bA+\bB K_d)(zI-A_d)^{-1}D=D.
\]
If \(D\) is strictly proper and FIR, the finite sum in \eqref{eq:constructive-repair-operator} is also strictly proper and FIR.  Multiplication by \(A_d^s\) expands the graph support by at most \(sr_d\), multiplication by \(K_d\) adds at most \(r_K\), and \(z^{-(s+1)}\) shifts the tap index by at most \(q\).  The largest spatial expansion is therefore \((q-1)r_d+r_K\).

For \(\diamond\in\{\row,\col\}\), the product inequalities give
\[
  \norm{\mathcal R(D)}_{\calL_1,\diamond}
  \le
  \left(1+\norm{K_d}_{\calL_1,\diamond}\right)
  \sum_{s=0}^{q-1}\norm{A_d^s}_{\calL_1,\diamond}
  \norm{D}_{\calL_1,\diamond}.
\]
Both prefactors are bounded by \((1+L_d)qL_d\) under Assumption~\ref{ass:local-deadbeat}.  Their maximum defines \(C_{\rm rep}\), which is independent of \(N\), and proves both norm estimates.
\end{proof}

\begin{proof}[Proof of Theorem~\ref{thm:conic-exact-local}]
Let \(r_0\) be the plant coupling range in Assumption~\ref{ass:lqr-local} and take
\[
c_{\rm sp}:=r_0+(q-1)r_d+r_K,
\qquad c_T:=q+1.
\]
For \(\theta=(\kappa,\bar\kappa,T)\) satisfying the stated lower bounds, define the eroded parameter triple
\[
  \theta^-:=(\kappa-c_{\rm sp},\bar\kappa-c_{\rm sp},T-c_T)
\]
and set \(\Psi:=\widehat\Phi_{\theta^-}^\star\), \(D:=\calZ\Psi-I\).  Let \(G_x=z\Psi_x-I\) and \(G_u=z\Psi_u\).  Then
\begin{equation}\label{eq:residual-filter-expansion}
  D=G_x-z^{-1}\bA(G_x+I)-z^{-1}\bB G_u.
\end{equation}
Because \(G_x[0]=0\), the residual \(D\) is strictly proper.

\emph{Residual generation.}
A coefficient at tap \(\tau\) of \(G_x\) or \(G_u\) contributes through \eqref{eq:residual-filter-expansion} to \(D\) at tap \(\tau\) or \(\tau+1\), with an increase in spatial radius of at most \(r_0\).  The standalone term \(-z^{-1}\bA\) occurs at tap one and has spatial radius at most \(r_0\).

\emph{Repair expansion.}
For \(0\le s\le q-1\), multiplication by \(A_d^s z^{-(s+1)}\) in \eqref{eq:constructive-repair-operator}, followed by conversion to the realization filters \(z\mathcal R_x(D)\) and \(z\mathcal R_u(D)\), adds at most \(q-1\) taps and \((q-1)r_d+r_K\) graph hops to a residual coefficient.  Relative to the originating coefficient of \(G_x\) or \(G_u\), every repaired filter coefficient therefore has a tap \(\tau'\) satisfying
\[
  \tau\le\tau'\le\tau+q
\]
and an increase in spatial radius of at most \(c_{\rm sp}\).

\emph{Containment in \(\calM_\theta\).}
Every originating coefficient of \(G_x\) or \(G_u\) lies in the eroded mask, so its repaired descendants have graph distance at most
\[
\min\{\bar\kappa-c_{\rm sp},\kappa-c_{\rm sp}+v_{\rm cone}\tau\}+c_{\rm sp}
\le \min\{\bar\kappa,\kappa+v_{\rm cone}\tau'\}.
\]
Since the eroded horizon is \(T-c_T\), the choice \(c_T=q+1\) places these coefficients within the horizon \(T\).  After repair, the standalone term \(-z^{-1}\bA\) has tap at most \(q\) and spatial radius at most \(c_{\rm sp}\), and hence also lies in \(\calM_\theta\).  Thus
\[
\Phi_\theta^{\rm rep}:=\Psi-\mathcal R(D)
\in\calS_\theta^{\rm impl},
\qquad
\calZ\Phi_\theta^{\rm rep}=I.
\]
This proves \(\calF_\theta^{\rm impl}\ne\emptyset\).  Proposition~\ref{prop:uniform-fir-repair}, \eqref{eq:sls-affine-bounded}, and the triangle inequality give
\[
\norm{\Phi_\theta^{\rm rep}-\Phi^\star}_{\calL_1,\row}
\le
(1+C_{\rm rep}L_{\rm sls})
\norm{\widehat\Phi_{\theta^-}^\star-\Phi^\star}_{\calL_1,\row}.
\]
Since \(J_\theta^{\rm loc}\le J_N(\Phi_\theta^{\rm rep})\), the Pythagorean identity \eqref{eq:pythagorean-cost}, the norm bridge \eqref{eq:weighted-l1-to-h2}, and Theorem~\ref{thm:tail} yield the upper bound.  Replacing \((\kappa,\bar\kappa,T)\) by the corresponding erosions with fixed offsets changes only the prefactor, and \((a+b+c)^2\le3(a^2+b^2+c^2)\) gives the stated rates.  The lower bound follows from centralized optimality.
\end{proof}

\subsubsection{Resource law independent of network size}\label{app:resource-law-proof}

\begin{proof}[Proof of Corollary~\ref{cor:accuracy-resources}]
Let
\[
  \lambda:=\min\{-\log\delta_\kappa,
  -\log\delta_{\bar\kappa},-\log\delta_T\}>0,
\]
where the rates are enlarged, if necessary, so that the residual and direct cost estimates use the same three rates.  Choose \(m_\varepsilon\in\Z_{\ge0}\) such that
\begin{equation}\label{eq:explicit-resource-parameter-choice}
  m_\varepsilon\ge
  \max\left\{
  \frac{1}{\lambda}\log\frac{3C_{\rm sg}}{\alpha_{\rm sg}},
  \frac{1}{2\lambda}\log\frac{3C_{\rm cl}}{\varepsilon}
  \right\}.
\end{equation}
With \(\kappa_\varepsilon=\bar\kappa_\varepsilon=T_\varepsilon=m_\varepsilon\), Lemma~\ref{lem:aux-tail-residual-estimates} gives \(\norm{\Delta_{\theta_\varepsilon}}_{\mathcal H_\infty}\le\alpha_{\rm sg}\), and Theorem~\ref{thm:direct-truncation-small-gain} gives \eqref{eq:accuracy-direct-gap}.  Taking the ceiling in \eqref{eq:explicit-resource-parameter-choice} produces constants \(c_0,c_1\), independent of \(N\), for which \eqref{eq:accuracy-parameter-scaling} holds.

Under Assumption~\ref{ass:local-deadbeat}, let \(\lambda_{\rm ex}>0\) be the minimum of the negative logarithms of the three rates in Theorem~\ref{thm:conic-exact-local}, and define
\begin{align*}
  m_\varepsilon^J&:=
  \left\lceil(2\lambda_{\rm ex})^{-1}\log(3C_J/\varepsilon)\right\rceil,\\
  m_\varepsilon^{\rm ex}&:=
  \max\{m_\varepsilon,c_{\rm sp},c_T,m_\varepsilon^J\}.
\end{align*}
Setting all three control architecture parameters to \(m_\varepsilon^{\rm ex}\) gives exact feasibility and \eqref{eq:accuracy-exact-gap}.  Fixed margins and constants only change \(c_0,c_1\).

Finally, the polynomial ball bound and \(d_{\max}(\tau)\le\bar\kappa_\varepsilon\) give
\[
  C_{\rm halo}\le C_{\rm gr}(1+\kappa_\varepsilon)^D,
  \quad
  C_{\rm foot}\le C_{\rm gr}(1+\bar\kappa_\varepsilon)^D,
\]
and
\[
  C_{\rm tap}
  \le (T_\varepsilon+1)C_{\rm gr}(1+\bar\kappa_\varepsilon)^D.
\]
Substitution of \eqref{eq:accuracy-parameter-scaling} proves \eqref{eq:accuracy-resource-scaling}.
\end{proof}

\end{document}

%% file: Figures/spacetime_decay_tikz.tex
\begin{tikzpicture}[
  >=Latex,
  font=\footnotesize,
  plane/.style={draw=black!27, line width=0.45pt, rounded corners=3pt, fill=black!1},
  stedge/.style={draw=black!80, line width=0.35pt, line cap=round},
  base node/.style={circle, fill=black, inner sep=0pt, minimum size=3.5pt},
  hopone/.style={circle, draw=#1!85!black, fill=#1!75, line width=0.25pt, inner sep=0pt, minimum size=3.3pt},
  hoptwo/.style={circle, draw=#1!45!black, fill=#1!28, line width=0.20pt, inner sep=0pt, minimum size=2.9pt},
  source ring/.style={circle, fill=none, line width=1.15pt, minimum size=11pt},
  channel/.style={densely dashed, line width=0.45pt, line cap=round},
  fan/.style={densely dashed, line width=0.28pt, line cap=round},
  column/.style={line width=5.2pt, line cap=round},
]
\newcommand{\STCoordinates}[1]{%
  \foreach \n/\x/\y in {
    0/-10.5/3.75,1/-9.25/6.25,2/-9.5/1.25,3/-6.25/7,
    4/-6.25/2,5/-4.25/3.75,6/-4/6,7/-5.5/9.25,
    8/-1.5/7.25,9/-1.75/9.5,10/1/6.75,11/-11/7.5,
    12/-11/9.75,13/-13.75/6.75,14/-11.5/-1,15/-13.5/0.5,
    16/-13.5/-3.75,17/-8.75/-1.75,18/-8.25/-4,19/-2/0.75,
    20/-1.25/3.75,21/1/2.25,22/-1.25/-2.25,23/2/-2.5,
    24/3.25/0.5,25/-3.25/-4.25,26/5.5/1.5,27/6/4.25}{%
    \coordinate (#1\n) at (\x,\y);
  }%
}
\newcommand{\STEdges}[1]{%
  \foreach \a/\b in {
    1/3,1/0,12/11,11/13,11/1,7/3,3/6,6/5,5/4,4/2,
    2/0,2/14,14/16,14/15,14/17,17/18,4/19,5/20,
    20/21,19/21,19/22,22/23,23/24,24/21,24/26,
    26/27,6/8,8/10,8/9,22/25}{%
    \draw[stedge] (#1\a) -- (#1\b);
  }%
}
\newcommand{\STBaseNodes}[1]{%
  \foreach \n in {0,...,27}{\node[base node] at (#1\n) {};}%
}
\newcommand{\STHopHighlights}[3]{%
  \begin{scope}[opacity=#3, transparency group]
    \foreach \a/\b in {2/0,2/4,2/14}{%
      \draw[red, line width=8.0pt, line cap=round] (#1\a) -- (#1\b);
    }%
    \foreach \a/\b in {21/20,21/19,21/24}{%
      \draw[blue, line width=8.0pt, line cap=round] (#1\a) -- (#1\b);
    }%
  \end{scope}
  \begin{scope}[opacity=#2, transparency group]
    \foreach \a/\b in {0/1,4/5,4/19,14/15,14/16,14/17}{%
      \draw[red, line width=12.0pt, line cap=round] (#1\a) -- (#1\b);
    }%
    \foreach \a/\b in {20/5,19/4,19/22,24/23,24/26}{%
      \draw[blue, line width=12.0pt, line cap=round] (#1\a) -- (#1\b);
    }%
  \end{scope}
}
\newcommand{\STColoredNodes}[3]{%
  \foreach \n in {0,4,14}{\node[hopone=red, opacity=#3] at (#1\n) {};}%
  \foreach \n in {1,5,15,16,17,19}{\node[hoptwo=red, opacity=#2] at (#1\n) {};}%
  \foreach \n in {20,19,24}{\node[hopone=blue, opacity=#3] at (#1\n) {};}%
  \foreach \n in {4,5,22,23,26}{\node[hoptwo=blue, opacity=#2] at (#1\n) {};}%
  \node[circle, fill=red!80!black, inner sep=0pt, minimum size=3.6pt, opacity=#3] at (#12) {};
  \node[circle, fill=blue!80!black, inner sep=0pt, minimum size=3.6pt, opacity=#3] at (#121) {};
}
\newcommand{\STLayer}[5]{%
  \begin{scope}[shift={(0,0,#2)}, scale=#3]
    \draw[plane]
      (-17.1,-6.2) -- (9.3,-6.2) -- (2.1,11.1) -- (-17.9,11.1) -- cycle;
    \STCoordinates{#1}
    \STHopHighlights{#1}{#4}{#5}
    \STEdges{#1}
    \STBaseNodes{#1}
    \STColoredNodes{#1}{#4}{#5}
  \end{scope}
}

\begin{scope}[
  shift={(3.25,0.00)},
  x={(0.43cm,0.00cm)},
  y={(0.10cm,0.115cm)},
  z={(-0.28cm,2.16cm)}
]

\STLayer{A}{0.00}{1.00}{0.22}{0.36}
\STLayer{B}{1.20}{0.86}{0.15}{0.26}
\STLayer{C}{2.05}{0.74}{0.10}{0.18}

\foreach \a/\b/\opone/\optwo in {A/B/0.145/0.080,B/C/0.105/0.060}{%
  \foreach \n in {2,0,4,14}{%
    \draw[red!70, opacity=\opone, column] (\a2) -- (\b\n);
  }%
  \foreach \n in {1,5,15,16,17,19}{%
    \draw[red!45, opacity=\optwo, column] (\a2) -- (\b\n);
  }%
  \foreach \n in {21,20,19,24}{%
    \draw[blue!70, opacity=\opone, column] (\a21) -- (\b\n);
  }%
  \foreach \n in {4,5,22,23,26}{%
    \draw[blue!45, opacity=\optwo, column] (\a21) -- (\b\n);
  }%
}

\foreach \a/\b/\opone/\optwo in {A/C/0.050/0.026}{%
  \foreach \n in {2,0,4,14}{%
    \draw[red!65, opacity=\opone, column] (\a2) -- (\b\n);
  }%
  \foreach \n in {1,5,15,16,17,19}{%
    \draw[red!40, opacity=\optwo, column] (\a2) -- (\b\n);
  }%
  \foreach \n in {21,20,19,24}{%
    \draw[blue!65, opacity=\opone, column] (\a21) -- (\b\n);
  }%
  \foreach \n in {4,5,22,23,26}{%
    \draw[blue!40, opacity=\optwo, column] (\a21) -- (\b\n);
  }%
}

\foreach \a/\b/\op in {A/B/0.30,B/C/0.22}{%
  \foreach \n in {2,0,4,14}{%
    \draw[red!70, fan, opacity=\op] (\a2) -- (\b\n);
  }%
  \foreach \n in {1,5,15,16,17,19}{%
    \draw[red!45, fan, opacity=\op] (\a2) -- (\b\n);
  }%
  \foreach \n in {21,20,19,24}{%
    \draw[blue!70, fan, opacity=\op] (\a21) -- (\b\n);
  }%
  \foreach \n in {4,5,22,23,26}{%
    \draw[blue!45, fan, opacity=\op] (\a21) -- (\b\n);
  }%
}
\foreach \a/\b/\op in {A/C/0.095}{%
  \foreach \n in {2,0,4,14}{%
    \draw[red!65, fan, opacity=\op] (\a2) -- (\b\n);
  }%
  \foreach \n in {1,5,15,16,17,19}{%
    \draw[red!40, fan, opacity=\op] (\a2) -- (\b\n);
  }%
  \foreach \n in {21,20,19,24}{%
    \draw[blue!65, fan, opacity=\op] (\a21) -- (\b\n);
  }%
  \foreach \n in {4,5,22,23,26}{%
    \draw[blue!40, fan, opacity=\op] (\a21) -- (\b\n);
  }%
}

\draw[red!75, channel, opacity=0.58] (A2) -- (B2);
\draw[red!75, channel, opacity=0.36] (B2) -- (C2);
\draw[blue!75, channel, opacity=0.58] (A21) -- (B21);
\draw[blue!75, channel, opacity=0.36] (B21) -- (C21);

\node[source ring, draw=red!85] at (A2) {};
\node[red!82, font=\large, fill=white, fill opacity=0.80, text opacity=1,
  inner sep=0.5pt, xshift=-12pt, yshift=-14pt] at (A2) {\(i\)};
\node[source ring, draw=blue!82] at (A21) {};
\node[blue!82, font=\large, fill=white, fill opacity=0.80, text opacity=1,
  inner sep=0.5pt, xshift=12pt, yshift=-14pt] at (A21) {\(j\)};

\end{scope}

\node[font=\large, xshift=58pt, yshift=-2pt] at (A10) {\(t=0\)};
\node[font=\large, xshift=58pt, yshift=-1pt] at (B10) {\(t=1\)};
\node[font=\large, xshift=58pt, yshift=0pt] at (C10) {\(t=2\)};

\draw[-{Latex[length=3.2mm]}, line width=0.75pt] (-4.65,0.10) -- (-3.10,4.95);
\node[font=\small\bfseries, rotate=76] at (-3.98,3.05) {time};

\end{tikzpicture}

%% file: Figures/three_parameter_architecture_tikz.tex
\begin{tikzpicture}[
  >=Latex,
  font=\small,
  axis/.style={black, line width=0.9pt, -{Latex[length=3.0mm]}},
  tick/.style={black, line width=0.8pt},
  support/.style={fill=black!16, draw=none},
  boundary/.style={blue!78, densely dashed, line width=1.18pt},
  measure/.style={red!88, line width=1.0pt, {Latex[length=2.6mm]}-{Latex[length=2.6mm]}},
  note/.style={align=center, text=black!86}
]
\def\tzero{1.25}
\def\tone{1.86}
\def\tsat{2.56}
\def\tend{4.55}
\def\ykappa{0.56}
\def\ykbar{1.42}
\def\yaxis{-2.30}

\draw[axis] (0,\yaxis) -- (6.05,\yaxis);
\draw[axis] (0,\yaxis) -- (0,2.48);
\node[font=\large, anchor=west] at (5.15,\yaxis-0.45) {Time};
\node[font=\large, anchor=east] at (-0.22,2.26) {Space};

\foreach \y/\lab in {
  \ykbar/{x_{i+\bar\kappa}},
  \ykappa/{x_{i+\kappa}},
  0/{x_i},
  -\ykappa/{x_{i-\kappa}},
  -\ykbar/{x_{i-\bar\kappa}}
}{
  \draw[tick] (-0.18,\y) -- (0.18,\y);
  \node[anchor=east] at (-0.32,\y) {\(\lab\)};
}

\foreach \x/\lab in {
  \tzero/{t},
  \tone/{t+1},
  \tend/{t+T}
}{
  \draw[tick] (\x,\yaxis-0.15) -- (\x,\yaxis+0.15);
  \node[anchor=north] at (\x,\yaxis-0.25) {\(\lab\)};
}

\path[support]
  (\tzero,\ykappa) --
  (\tsat,\ykbar) --
  (\tend,\ykbar) --
  (\tend,-\ykbar) --
  (\tsat,-\ykbar) --
  (\tzero,-\ykappa) --
  cycle;
\draw[boundary] (\tzero,\ykappa) -- (\tzero,-\ykappa);
\draw[boundary] (\tzero,\ykappa) -- ({\tsat+0.58},{\ykbar+0.38});
\draw[boundary] (\tzero,-\ykappa) -- ({\tsat+0.58},{-\ykbar-0.38});
\draw[boundary] ({\tsat-0.62},\ykbar) -- ({\tend+0.66},\ykbar);
\draw[boundary] ({\tsat-0.62},-\ykbar) -- ({\tend+0.66},-\ykbar);
\draw[boundary] (\tend,{\ykbar+0.70}) -- (\tend,{-\ykbar-0.70});

\draw[black!35, densely dotted, line width=0.55pt] (\tzero,0) -- (\tend,0);
\node[note] at (3.15,0.0) {Allowed filter\\support};

\fill[red!88] (\tzero,0) circle (2.6pt);
\node[red!88, anchor=south west] at ({\tzero+0.10},0.10) {\(\widehat w_i\)};

\draw[measure] (\tzero-0.24,-\ykappa) -- (\tzero-0.24,0);
\node[red!88, anchor=east] at (\tzero-0.38,-0.28) {\(\kappa\)};

\draw[measure] (\tend+0.34,0) -- (\tend+0.34,\ykbar);
\node[red!88, anchor=west] at (\tend+0.50,0.72) {\(\bar\kappa\)};

\draw[measure] (\tzero,\yaxis+0.30) -- (\tend,\yaxis+0.30);
\node[red!88, fill=white, inner sep=1.0pt] at ({0.5*(\tzero+\tend)},\yaxis+0.48)
  {\(T\)};
\end{tikzpicture}

%% file: Figures/kkt_response_graph_definition_tikz.tex
\begin{tikzpicture}[
  x=1cm,
  y=1cm,
  font=\scriptsize,
  state/.style={draw=blue!65!black,rounded corners=1.2pt,minimum width=13mm,minimum height=5.5mm,inner sep=1.5pt,fill=blue!6},
  input/.style={draw=green!45!black,rounded corners=1.2pt,minimum width=13mm,minimum height=5.5mm,inner sep=1.5pt,fill=green!8},
  mult/.style={draw=orange!75!black,rounded corners=1.2pt,minimum width=14mm,minimum height=5.5mm,inner sep=1.5pt,fill=orange!10},
  kktedge/.style={thin,draw=black!65},
  xpath/.style={semithick,draw=blue!65!black},
  upath/.style={semithick,draw=green!45!black},
  response/.style={-{Latex[length=1.5mm]},semithick}
]

\node[anchor=west,font=\scriptsize\bfseries] at (0,3.70) {(a)};

\node[state] (xi) at (0.90,3.05) {$X_i(s)$};
\node[input] (ui) at (0.90,2.15) {$U_i(s)$};
\node[state] (xj) at (3.35,3.05) {$X_j(s)$};
\node[input] (uj) at (3.35,2.15) {$U_j(s)$};
\node[mult] (lami) at (5.70,2.60) {$\Lambda_i(s)$};
\node[state] (xinext) at (8.00,2.60) {$X_i(s+1)$};

\node[draw=black!50,dashed,rounded corners=2pt,fit=(xi)(xj)(ui)(uj)(lami),
  inner xsep=2.5mm,inner ysep=2.5mm] (cell) {};
\node[fill=white,inner sep=0.7pt] at (3.30,3.58) {time $s$};

\draw[kktedge] (xi) -- node[above=0.6mm,fill=white,inner sep=0.5pt]
  {$Q_{ij}\ne0$} (xj);
\draw[kktedge] (ui) -- node[above=0.6mm,fill=white,inner sep=0.5pt]
  {$R_{ij}\ne0$} (uj);
\draw[kktedge] (lami) -- node[pos=0.50,above=0.6mm,fill=white,inner sep=0.5pt]
  {$A_{ij}\ne0$} (xj);
\draw[kktedge] (lami) -- node[pos=0.50,below=0.6mm,fill=white,inner sep=0.5pt]
  {$B_{ij}\ne0$} (uj);
\draw[kktedge] (lami) -- node[pos=0.50,above=0.7mm,fill=white,inner sep=0.5pt]
  {update} (xinext);
\node[anchor=west] at (8.82,2.60) {$\cdots$};

\node[anchor=west,font=\scriptsize\bfseries] at (0,1.35) {(b)};

\node[mult] (source) at (0.85,0.45) {$\Lambda_j(-1)$};
\node[state] (xzero) at (2.55,0.45) {$X_j(0)$};
\node[inner sep=0pt] (xdots) at (4.10,0.82) {$\cdots$};
\node[inner sep=0pt] (udots) at (4.10,0.08) {$\cdots$};
\node[state] (xtarget) at (5.55,0.82) {$X_i(t)$};
\node[input] (utarget) at (5.55,0.08) {$U_i(t)$};

\draw[kktedge] (source) -- node[above=0.9mm,fill=white,inner sep=0.5pt]
  {initial} (xzero);
\draw[xpath] (xzero) -- (xdots) -- (xtarget);
\draw[upath] (xzero) -- (udots) -- (utarget);

\node[anchor=west,text=blue!65!black] (phix) at (6.65,0.82)
  {$\Phi_x^\star[t+1]_{ij}$};
\node[anchor=west,text=green!45!black] (phiu) at (6.65,0.08)
  {$\Phi_u^\star[t+1]_{ij}$};
\draw[response,draw=blue!65!black] (xtarget) -- (phix.west);
\draw[response,draw=green!45!black] (utarget) -- (phiu.west);

\end{tikzpicture}